\documentclass[12pt]{amsart}
\usepackage{amssymb,amsmath, amsthm,latexsym}
\usepackage{a4wide}
\theoremstyle{plain}

\newtheorem{lemma}{Lemma}[section]
\newtheorem{theo}[lemma]{Theorem}
\newtheorem{prop}[lemma]{Proposition}
\newtheorem{corollary}[lemma]{Corollary}
\theoremstyle{remark}
\newtheorem{remark}{Remark}%[subsection]

\parskip=\medskipamount

\font\rm=cmr12

\newcommand{\iii}{{\mathrm i}}
\newcommand{\Ht}{H}
\newcommand{\vs}{{\mathfrak a}}
\newcommand{\Ka}{K_\infty^0}
\newcommand{\cu}{\operatorname{cus}}
\newcommand{\di}{\operatorname{dis}}

\newcommand{\spec}{\operatorname{spec}}
\newcommand{\geo}{\operatorname{geo}}
\newcommand{\nm}{{\mathcal N}}
\newcommand{\unip}{\operatorname{unip}}
\newcommand{\untry}{\operatorname{un}}
\newcommand{\hm}{\operatorname{hm}}
\newcommand{\mult}{m}
\newcommand{\invh}{{\mathcal B}}
\newcommand{\FFF}{{\mathcal F}}
\newcommand{\C}{{\mathbb C}}
\newcommand{\N}{{\mathbb N}}
\newcommand{\R}{{\mathbb R}}
\newcommand{\Z}{{\mathbb Z}}
\newcommand{\Q}{{\mathbb Q}}
\newcommand{\A}{{\mathbb A}}
\newcommand{\bH}{{\mathbb H}}
\newcommand{\varbeta}{\tilde\beta}

\newcommand{\gf}{{\mathfrak g}}
\newcommand{\mf}{{\mathfrak m}}
\newcommand{\kf}{{\mathfrak k}}
\newcommand{\ho}{{\mathfrak o}}
\newcommand{\pg}{{\mathfrak p}}
\newcommand{\mM}{{\mathfrak M}}
\newcommand{\mN}{{\mathfrak N}}
\newcommand{\HHH}{{\mathcal H}}
\newcommand{\M}{{\mathcal M}}
\newcommand{\Latt}{{\mathcal Z}}

\newcommand{\cO}{{\mathcal O}}
\newcommand{\pars}{{\mathcal P}}
\newcommand{\leviP}{{\mathcal F}}
\newcommand{\levis}{{\mathcal L}}
\newcommand{\cA}{{\mathcal A}}
\newcommand{\AF}{{\mathcal A}}
\newcommand{\cU}{{\mathcal U}}
\newcommand{\bs}{\backslash}
\newcommand{\sprod}[2]{\left\langle#1,#2\right\rangle}
\newcommand{\ov}{\overline}
\renewcommand{\Im}{\operatorname{Im}}
\renewcommand{\Re}{\operatorname{Re}}
\newcommand{\Ind}{\operatorname{Ind}}

\newcommand{\tr}{\operatorname{tr}}

\newcommand{\vol}{\operatorname{vol}}
\newcommand{\Area}{\operatorname{Area}}
\newcommand{\SL}{\operatorname{SL}}
\newcommand{\GL}{\operatorname{GL}}
\newcommand{\SO}{\operatorname{SO}}
\newcommand{\Ad}{\operatorname{Ad}}
\newcommand{\supp}{\operatorname{supp}}
\newcommand{\nnn}{{\mathfrak n}}
\newcommand{\ppp}{{\mathfrak p}}
\newcommand{\orbit}{{\mathcal O}}

\newcommand{\kkk}{{\mathfrak k}}
\newcommand{\card}[1]{\lvert#1\rvert}
\newcommand{\abs}[1]{\lvert#1\rvert}
\newcommand{\norm}[1]{\lVert#1\rVert}
\newcommand{\eps}{\epsilon}
\newcommand{\ad}{\operatorname{ad}}
\newcommand{\proj}{\operatorname{proj}}

\title[Spectral asymptotics]
{Spectral asymptotics for arithmetic quotients of $\SL(n,\R)/\SO(n)$}
\author{Erez Lapid}
\address{Institute of Mathematics\\
Hebrew University\\
Givat Ram\\
Jerusalem, Israel 91904}
\email{erezla@math.huji.ac.il}
\author{Werner M\"uller}
\address{Universit\"at Bonn\\
Mathematisches Institut\\
Beringstrasse 1\\
D -- 53115 Bonn, Germany}
\email{mueller@math.uni-bonn.de}

\thanks{This research was sponsored by GIF grant \# I-796-167.6/2003}
\date{\today}

\begin{document}

\begin{abstract}
In this paper we study the asymptotic distribution of the cuspidal spectrum 
of arithmetic quotients of the symmetric space $\SL(n,\R)/\SO(n)$. In 
particular, we obtain Weyl's law with an estimation on the remainder term.
This extends some of the main results of Duistermaat-Kolk-Varadarajan
(\cite{MR532745}) to this setting.

\end{abstract}

\maketitle

\setcounter{tocdepth}{1}
\tableofcontents

\section{Introduction}

Let $G$ be a reductive algebraic group over $\Q$, $\A$ the ring of adeles of 
$\Q$, and $\omega$ a unitary character of $Z(\Q)\bs Z(\A)$ where $Z$ is the 
center of $G$. One of the fundamental problems in the theory of automorphic
forms is to determine the spectral decomposition of the regular representation 
of $G(\A)$ on
$L^2(G(\Q)\bs G(\A),\omega)$. Langlands' theory essentially reduces
the problem to the discrete part.
There are deep conjectures of Arthur (\cite{MR1110389}) which are aiming at
the description of the discrete spectrum.
However, these conjectures are out of reach at present.  

One of the basic tools to study these problems is Arthur's trace formula. The main 
driving force behind Arthur's approach is the functoriality conjectures
of Langlands. Consequently, the ability to compare the trace formula on two
different groups is a key issue in Arthur's setup.

On the other hand, it is natural to ask what spectral information on the group
itself can be inferred from the trace formula in higher rank.
For example, it is well-known that the dimension of the space of
automorphic forms with
certain square-integrable Archimedean components
can be computed using the trace formula (\cite{MR0156362}).
More generally, there is an exact formula for traces of Hecke operators
on automorphic forms whose Archimedean component lie in a discrete series
$L$-packet with a non-singular Harish-Chandra parameter (\cite{MR1001841}).
(Cf.~\cite{MR1470341} for a geometric counterpart.)

In the realm of spectral theory, a basic problem is to study the
asymptotic distribution of the infinitesimal characters of the Archimedean 
components of cusp forms with a fixed $K_\infty$-type. In the 
simplest case one counts the Casimir eigenvalues of the Archimedean components
of cusp forms. In fact, this was Selberg's motivation for developing the
trace formula, with which he established the analogue of Weyl's law for the 
cuspidal
spectrum of the quotient of the hyperbolic plane by a congruence subgroup
\cite{MR0088511}, \cite[p.~626--674]{MR1117906}.
This was extended to other rank one locally symmetric spaces by Reznikov 
\cite{MR1204788}.
In higher rank, Duistermaat-Kolk-Varadarajan proved quite general results
about the asymptotic distribution of the spherical spectrum
for compact locally symmetric spaces \cite{MR532745}
and in particular gave an upper bound on the complementary spectrum.
(Weyl's law itself, with a sharp remainder term, was proved earlier for the 
Laplace operator of 
any compact Riemannian manifold by Avakumovi\'c 
\cite{MR0080862} and this was generalized to any elliptic pseudo-differential 
operator by H\"ormander \cite{MR0609014}.)
The first example of a non-uniform lattice in higher rank
was treated by S. Miller \cite{MR1823867} who proved that for $\SL(3,\Z)$, the
cuspidal spectrum satisfies Weyl's law and the tempered cuspidal spectrum has
density $1$.
%In the non-compact case, Weyl's law and the fact that the tempered spectrum
%has density $1$ were proved by Miller for $\SL(3,\Z)\bs\SL(3,\R)/\SO(3)$
%(\cite{MR1823867}).
Analogues of the Weyl's law for an arbitrary $K_\infty$-type were
obtained by the second-named author for arithmetic quotients of 
$\GL(n)$, $n\ge2$ in \cite{MR2276771}.
For the spherical spectrum, Lindenstrauss and Venkatesh (\cite{LV}) showed 
that Weyl's law holds
in great generality, proving a conjecture of Sarnak. 

Our main goal in this paper is to extend the results of \cite{MR532745}
(with a slightly weaker error term) to the case of
arithmetic quotients of the symmetric space $X$ of positive-definite 
quadratic forms in $n\ge2$ variables up to homothety. 
In particular, we derive Weyl's law with a remainder term
strengthening the result of \cite{MR2276771} (for the spherical case).

Now we describe our results in more detail.
Let $G=\GL(n)$. Let $A_G$ be the group of scalar matrices with a positive
real scalar, $W$ the Weyl group  and $K_\infty=\mathrm{O}(n)$.
%Let $G(\A)$ be the group of adeles.
%Let $\Pi(G(\A))$ be the set of equivalence classes of irreducible unitary
%admissible representations of $G(\A)$ whose central character is trivial on $A_G$.
%It is well known that the multiplicity of $\pi$ in the discrete part
%of $L^2(A_GG(\Q)\bs G(\A))$ is at most one (\cite{MR618323}, \cite{MR623137}, 
%\cite{MR1026752}).
Let $\Pi_{\cu}(G(\A))$ (resp.~$\Pi_{\di}(G(\A))$) denote the set of 
irreducible unitary representations of $G(\A)$ occurring in the cuspidal (resp.~discrete) spectrum
of $L^2(A_GG(\Q)\bs G(\A))$ (necessarily with multiplicity one,
cf.~\cite{MR618323}, \cite{MR623137}, \cite{MR1026752}).
Let $\vs^*=\{(\lambda_1,\dots,\lambda_n)\in\R^n:\sum\lambda_i=0\}$.
Given $\pi\in\Pi_{\di}(G(\A))$, denote by
$\lambda_{\pi_\infty}\in\vs_{\C}^*/W$ the infinitesimal character of the
Archimedean component $\pi_\infty$ of $\pi$ and for any subgroup $K$ of $G(\A)$
let $\HHH_\pi^K$ be the space of $K$-invariant vectors in the representation
space of $\pi$. We refer to \S\ref{prelim} for normalization of various
Haar measures and the Plancherel measure $\beta(\lambda)$
on $\iii\vs^*$ pertaining to $A_G\bs G(\R)$.
The adelic version of our main result is the following Theorem. 
\begin{theo}\label{mainthm}
Let $K=K_\infty K_f$ where $K_f$ is an open compact subgroup of $G(\A_f)$ 
contained in some principal congruence subgroup $K_f(N)$ of level $N\ge3$,
and let $\Omega\subseteq\iii\vs^*$ be a $W$-invariant bounded domain with piecewise $C^2$ boundary.
Then
\[
\sum_{\substack{\pi\in\Pi_{\cu}(G(\A))\\
\lambda_{\pi_\infty}\in t\Omega}}\dim\left(\HHH_\pi^K\right)=
\frac{\vol(A_GG(\Q)\bs G(\A)/K_f)}{\card{W}}
\int_{t\Omega}\beta(\lambda)\ d\lambda+
O\left(t^{d-1}(\log t)^{\max(n,3)}\right)
\]
as $t\to\infty$
where $t\Omega=\{t\lambda\colon\lambda\in\Omega\}$ and $d=\dim X$.

On the other hand for the ball $B_t(0)$ of radius $t$ in $\vs_\C^*$ centered at the origin
we have
\[
\sum_{\substack{\pi\in\Pi_{\di}(G(\A))\\
\lambda_{\pi_\infty}\in B_t(0)\setminus\iii\vs^*}}\dim\left(\HHH_\pi^K\right)=
O(t^{d-2})
\]
as $t\to\infty$. 
\end{theo}
Thus, the complementary cuspidal spectrum, which, according to the
Archimedean aspect of the Ramanujan conjecture for $\GL(n)$, is not expected to exist at all,
is at least of lower order of magnitude compared to the tempered spectrum.
Theorem \ref{mainthm} (with an appropriate power of $\log t$,
or perhaps even without it) is expected to hold for any reductive group $G$
over $\Q$ (where cuspidal spectrum is not necessarily tempered in general).
Although the method of proof in principle carries over to any $G$,
there are some issues which at present we do not know how to deal with in general.
%However, at this stage we can only work with $\GL(n)$
%due to several special features of it. (See below.)

One can rephrase Theorem \ref{mainthm} in more classical terms.
Note that $X=A_G\bs G(\R)/K_\infty=\SL(n,\R)/\SO(n)$.
Let $\Gamma(N)\subseteq\SL(n,\Z)$ be the principal congruence subgroup of level
 $N$. 
Let $\Lambda_{\cu}(\Gamma(N))\subseteq\vs_{\C}^*/W$ be the cuspidal spectrum
of the algebra of invariant differential operators of $\SL(n,\R)$, acting  in 
$L^2(\Gamma(N)\bs X)$. Given $\lambda\in\Lambda_{\cu}(\Gamma(N))$, denote by
$m(\lambda)$ the dimension of the corresponding eigenspace.

\begin{corollary} \label{maincor}
For $N\ge 3$ and $\Omega$ as before we have
\begin{equation} \label{classform1}
\sum_{\lambda\in\Lambda_{\cu}(\Gamma(N)),
\lambda\in t\Omega}m(\lambda)=
\frac{\vol(\Gamma(N)\bs X)}{\card{W}}\int_{t\Omega}\beta(\lambda)\ d\lambda+
O\left(t^{d-1}(\log t)^{\max(n,3)}\right)
\end{equation}
and
\begin{equation} \label{classform2}
\sum_{\substack{\lambda\in\Lambda_{\cu}(\Gamma(N))\\
\lambda\in B_t(0)\setminus\iii\vs^*}}m(\lambda)=
O\left(t^{d-2}\right)
\end{equation}
as $t\to\infty$.
\end{corollary}

The symmetric space $X$ is endowed with a Riemannian structure defined using
the Killing form.
Taking the corresponding volume form and Laplacian operator $\Delta$ and
applying Corollary \ref{maincor} to the unit ball in $\iii\vs^*$ we get
\begin{corollary} \label{maincor2}
Let 
\[
N_{\cu}^{\Gamma(N)}(t)=\#\{j\colon\lambda_j\le t^2\}
\]
be the counting function for the cuspidal eigenvalues 
$0<\lambda_1\le\lambda_2\le\dots$ (counted with multiplicity)
of $\Delta$ on $\Gamma(N)\bs X$, $N\ge 3$. Then
\[
N_{\cu}^{\Gamma(N)}(t)=\frac{\vol(\Gamma(N)\bs X)}
{(4\pi)^{d/2}{\bf\Gamma}(d/2+1)}t^d+
O\left(t^{d-1}(\log t)^{\max(n,3)}\right)
\]
as $t\to\infty$.
\end{corollary}

The condition $N\ge3$ is imposed for technical reasons. It guarantees that
the principal congruence subgroup $\Gamma(N)$ is \emph{neat} in the sense of Borel (\cite{MR0244260}),
and in particular acts freely on $X$.
This simplifies the analysis by eliminating the contribution of the non-unipotent conjugacy classes.

Our proof uses the method of Duistermaat-Kolk-Varadarajan (\cite{MR532745})
who proved the same result (without the $\log$ factor in the error term)
for any \emph{compact} locally symmetric space.
This simplifies and strengthens the argument of \cite{MR2276771}
where the heat kernel was used instead.

In a nutshell, one has to show that for an appropriate family of test functions
the main contribution in the trace formula comes in the geometric side
from the identity conjugacy class, and in the spectral side from the
cuspidal spectrum.
Of course, in our case we have to consider the trace formula in Arthur's
(non-invariant) form.

The main new technical difficulty is the analysis of the contribution of the unipotent
conjugacy classes (\cite{MR828844}).
These distributions are weighted orbital integrals and they can be analyzed 
by the 
method of stationary phase, just like the ordinary orbital integrals
(cf.~\cite{MR707179} for the semi-simple case).
A mitigating factor is that all unipotent orbits in the case of $\GL(n)$
are of Richardson type, which somewhat simplifies the structure of
the integral expression for these distributions. 

On the spectral side, we have to show that the contribution of the 
continuous spectrum is of a lower order of magnitude.
(This point is non-trivial in general, because it is rarely expected to hold
for non-arithmetic non-uniform lattices (\cite{MR812352}, \cite{MR853570}, 
\cite{MR1127079}).)
This is done as in \cite{MR2276771} by controlling this contribution in terms of
Rankin-Selberg $L$-functions and using the analytic properties of the latter.

In contrast, the method of \cite{LV} completely avoids the contribution
of the continuous spectrum by choosing appropriate test functions.
This is an ingenious extension of the so-called simple trace formula
(cf.~\cite{MR771672})
where in its ordinary form the continuous spectrum is removed at the price of a positive
proportion of the cuspidal spectrum (\cite{MR2127945}). 
It would be interesting to see whether the ideas of \cite{LV} can be pushed
further to give tight upper bounds on the size of the continuous spectrum
without appealing to the theory of $L$-functions 
(which is not always available). 

The contents of the paper are as follows.
We first give an outline of the proof of the Weyl law with remainder for the case of
$\SL_2$ in \S\ref{SL2} by combining H\"ormander's method and the Selberg trace formula.
The proof is not required for the higher rank case, but is given
as a precursor for the general case (which is technically more
difficult), in which
H\"ormander's method is replaced by the method of Duistermaat-Kolk-Varadarajan
and the Selberg trace formula is substituted by Arthur's trace formula.
In \S\ref{prelim} we recall some basic facts about harmonic analysis on $X$ and
the spherical unitary dual and introduce some notation.
The main section is \S\ref{DKV} where we explain the higher rank setup and
reduce the problem to an estimation of the contributions to the trace
formula of the non-trivial
conjugacy classes on the one hand (to the geometric side) and the non-discrete part
of the spectrum on the other hand (to the spectral side).
These estimations are carried out in sections \ref{geom} and \ref{spec}
respectively.
In the former we use the stationary phase method in its crudest form -
that is, a simple application of the divergence theorem.
In the latter we modify the analysis of \cite{MR2276771} to our setup.

It is natural to generalize our results to give the asymptotic
behavior of traces of Hecke operators on the space of Maass forms
with a uniform error term. This will have applications, among other things,
to the distribution of low-lying zeros of automorphic $L$-functions of $\GL(n)$
(cf.~\cite{MR1828743} for the case $n=2$).
We hope to pursue this in a forthcoming paper.

\subsection*{Acknowledgement}
We are grateful to Peter Sarnak for motivating us to work on this project and
for generously sharing some key ideas with us.
We would also like to thank G\"unter Harder, Werner Hoffmann and Michael
Rapoport for some helpful discussions and Stephen Miller for some
useful remarks.
Part of this work was done while the authors were visiting
the Institute for Advanced Study in Princeton. The authors thank the IAS
for its hospitality.

\section{The $\SL(2)$ case} \label{SL2}
\setcounter{equation}{0}

In this section we describe the main idea of the proof for the special case
of $\Gamma\bs\bH$. This has also been discussed in \cite{Mu}.
For the convenience of the reader we recall the main steps.
The  case of a hyperbolic surface  was first treated by Selberg in
\cite[p.~668]{MR1117906} (cf.~also \cite{MR0439755}).
Our method is a combination of H\"ormander's method
\cite{MR0609014} and the Selberg trace formula.

Let $\Gamma\subseteq\SL(2,\Z)$ be a congruence subgroup and let
$\Delta$ be the Laplacian of the hyperbolic surface $\Gamma\bs\bH$. Let 
$\lambda_0=0<\lambda_1\le\lambda_2\le\cdots$ be the eigenvalues of $\Delta$
acting in $L^2(\Gamma\bs\bH)$. We write
$\lambda_j=\frac14+r_j^2$ with $r_j\in\R_{\ge0}\cup[0,\frac12]\iii$.
For $\lambda\ge 0$ let 
\[
N^\Gamma(\lambda)=\#\{j:\abs{r_j}\le \lambda\}
\]
be the counting function for the eigenvalues, where each eigenvalue is counted 
with its multiplicity. Note that any eigenfunction with eigenvalue 
$\lambda_j\ge 1/4$ is cuspidal. Therefore $N^\Gamma(\lambda)
=N^\Gamma_{\cu}(\lambda)+C$
and  it suffices to study the asymptotic behavior of $N^\Gamma(\lambda)$.

We will write $O(X)$ for any quantity which is bounded in absolute value
by a constant multiple of $X$.
Sometimes we write $O_p(X)$ to indicate dependence of the implied constant
on additional parameters, but we often suppress it if it is understood
from the context.

Next recall the Selberg trace formula \cite{MR0088511}. 
Let $E_k(z,s)$ be the Eisenstein series attached to the $k$-th cusp. The 
constant term of its Fourier expansion in the $l$-th cusp is of the form 
\[
\delta_{kl}y^s+C_{kl}(s)y^{1-s}.
\]
Let $C(s)=(C_{kl}(s))$ be the corresponding scattering matrix
and let $\phi(s)=\det C(s)$. 
Let $h\in C^\infty_c(\R)$ and $\hat h(z)=\int_\R h(r)e^{irz}\;dr$. Then $\hat h$ is
entire and rapidly decreasing on horizontal strips. For $t\in\R$ put
\[
\hat h_t(z)=\hat h(t-z)+\hat h(t+z).
\]
Symmetrize the spectrum by $r_{-j}:=-r_j$, $j\in\N$.
For simplicity we assume that $\Gamma$ contains no elements of finite order.
Then the trace formula can be applied to $\hat h_t$ and gives the following identity.
\begin{multline}\label{selberg}
\sum_{j=-\infty}^\infty \hat h(t-r_j)=\frac{\Area(\Gamma\bs \bH)}{2\pi}
\int_\R \hat h(t-r)r\tanh(\pi r)\;dr+\sum_{\{\gamma\}_\Gamma} 
\frac{l(\gamma_0)}{\sinh \left( \frac{l(\gamma)}{2}\right)} h(l(\gamma))
\cos(t l(\gamma))\\
+\frac1{2\pi} \int^\infty_{-\infty}\hat h(t-r)\frac{\phi'}{\phi}(\frac12+\iii r)\;dr  
-\frac12\phi(\frac12)\hat h(t)
-\frac{m}{\pi}\int^\infty_{-\infty} \hat h(t-r)
\frac{{\bf\Gamma}'}{\bf\Gamma}(1+\iii r)dr +\frac m2\hat h(t)-\\2m\ln 2\; h(0).
\end{multline}
Here $m$ is the number of cusps of $\Gamma\bs\bH$ and $\{\gamma\}_\Gamma$ runs 
over the hyperbolic conjugacy classes in $\Gamma$. Each such conjugacy class
determines a closed geodesic $\tau_\gamma$ of $\Gamma\bs\bH$ and $l(\gamma)$ 
denotes the length of $\tau_\gamma$. Each hyperbolic element $\gamma$ is the 
power of a primitive hyperbolic element $\gamma_0$. Since $h\in C_c^\infty(\R)$,
all series and integrals are absolutely convergent. 

We use this formula to study the asymptotic behavior of the left hand side 
as $t\to\infty$. To this end we need to consider the asymptotic behavior of the
terms on the right hand side.
Using $\abs{\tanh(x)}\le 1$, $x\in\R$, it follows that
\[
\int_\R \hat h(t-r)r\tanh(\pi r)\;dr=O(t),\quad \abs{t}\to\infty.
\]
For the second term we observe that
the sum over the hyperbolic conjugacy classes is 
absolutely convergent. Hence the sum is uniformly bounded in $t$. 

Next consider the integral involving the scattering matrix. For the 
principal congruence subgroup $\Gamma(N)$ the determinant of the scattering 
matrix $\phi(s)=\det C(s)$ has been computed by Huxley \cite{MR0439755}. It has the 
form
\[
\phi(s)=(-1)^lA^{1-2s}\left(\frac{{\bf\Gamma}(1-s)}{{\bf\Gamma}(s)}\right)^k
\prod_\chi\frac{L(2-2s,\bar\chi)}{L(2s,\chi)},
\]
where $k,l\in\Z$, $A>0$, the product runs
over Dirichlet characters $\chi$ to some modulus dividing $N$ and $L(s,\chi)$
the Dirichlet $L$-function with character $\chi$. Using Stirling's 
approximation formula to estimate the logarithmic derivative of the Gamma 
function and standard estimations for the logarithmic derivative
of Dirichlet $L$-functions on the line $\Re(s)=1$, %e.g. ?
we get
\begin{equation}\label{philog}
\frac{\phi'}{\phi}(\frac12+\iii r)=O(\log(\abs{r})),\quad \abs{r}\to\infty.
\end{equation}
This implies
\[
\int_\R \hat h(t-r)\frac{\phi'}{\phi}(\frac12+\iii r)\;dr=O(\log(\abs{t})),\quad
\abs{t}\to\infty. 
\]
In the same way we get
\begin{equation} \label{logamma}
\int_\R \hat h(t-r)\frac{{\bf\Gamma}'}{{\bf\Gamma}}(1+\iii r)\;dr=O(\log(\abs{t})),\quad
\abs{t}\to\infty. 
\end{equation}
The remaining terms are bounded as $\abs{t}\to\infty$. Summarizing, we obtain
\begin{equation}\label{tmpbnd2}
\sum_{j=-\infty}^\infty \hat h(t-r_j)=O(\abs{t}),\quad \abs{t}\to\infty.
\end{equation}
This result can be applied to estimate the number of eigenvalues in a 
neighborhood of a given point $\lambda\in\R$.
As in the proof of \cite[Lemma 2.3]{MR0423438} choose 
$h\in C_c^\infty(\R)$ such that $\hat h\ge 0$ and $\hat h>0$ on $[-1,1]$.
Now note that there are only finitely many eigenvalues 
$\lambda_j=1/4+r^2_j$ with $r_j\notin\R$. Hence it suffices to consider the
eigenvalues with $r_j\in\R$. For $\lambda\in\R$ we have
\[
\#\{j:\abs{r_j-\lambda}\le 1,\;r_j\in\R\}\cdot 
\min\{\hat h(u):\abs{u}\le 1\}\le
\sum_{r_j\in\R} \hat h(\lambda-r_j)
\]
which is $O(1+\abs{\lambda})$ by (\ref{tmpbnd2}). It follows that
\[
\#\{j:\abs{r_j-\lambda}\le 1\}=O(1+\abs{\lambda})
\]
for all $\lambda\in\R$. This local estimation is the basis of the following 
auxiliary result.
Let $h$ be even. Then
\begin{equation}\label{locbnd}
\sum_{\abs{r_j}\le\lambda}\bigg|\int_{\R-[-\lambda,\lambda]} \hat h(t-r_j)\;dt\bigg|+
\sum_{\abs{r_j}>\lambda}\bigg|\int_{-\lambda}^\lambda \hat h(t-r_j)\;dt\bigg|=
O(\lambda)
\end{equation}
for all $\lambda\ge 1$. The proof is elementary (see \cite[Lemma 2.3]{Mu}).

The next step is to integrate both sides of (\ref{selberg}) over a finite
interval $(-\lambda,\lambda)$ and study the asymptotic behavior as 
$\lambda\to\infty$ of the terms on the right hand side. To this end let 
$p(r)$ be a continuous even function on $\R$ such that $p(r)=O(1+\abs{r})$.
Assume that $h(0)=1$. Then it follows that
\[
\int_{-\lambda}^\lambda \int_\R \hat h(t-r)p(r)\;dr\;dt=\int_{-\lambda}^\lambda
p(r)\;dr+O(\lambda),\quad \lambda\to\infty.
\]
\cite[(2.10)]{Mu}.
Applying this to the functions $p(r)=r\tanh(\pi r)$, $\frac{\phi'}{\phi}(\frac12+\iii r)$ and 
$\frac{{\bf\Gamma}'}{\bf\Gamma}(1+\iii r)$ respectively and using (\ref{philog}) and Stirling's
formula we obtain
\begin{gather*}
\int_{-\lambda}^\lambda \int_\R \hat h(t-r)r\tanh(\pi r)\;dr\;dt= \lambda^2+
O(\lambda),\\
\int_{-\lambda}^\lambda\int_\R \hat h(t-r)\frac{\phi'}{\phi}(\frac12+\iii r)\;dr\;dt=
O(\lambda\log\lambda),\\
\int_{-\lambda}^\lambda\int_\R \hat h(t-r)\frac{{\bf\Gamma}'}{\bf\Gamma}(1+\iii r)\;dr\;dt=
O(\lambda\log\lambda).
\end{gather*}
The remaining terms on the right hand side of (\ref{selberg}) stay bounded as
$t\to\infty$. Hence their integral is of order $O(\lambda)$. Thus, it follows that
for every even $h$ such that $h(0)=1$, we have 
\begin{equation}\label{tmpbnd}
\int_{-\lambda}^\lambda \sum_{j=-\infty}^\infty \hat h(t-r_j)\;dt=
\frac{\Area(\Gamma\bs\bH)}{2\pi}\lambda^2+O(\lambda\log \lambda)
\end{equation}
as $\lambda\to\infty$.

We are now ready to prove Weyl's law.
We choose $h$ with $h(0)=1$. Then the left-hand side of (\ref{tmpbnd}) is
\[
\sum_{\abs{r_j}\le\lambda}\int_\R \hat h(t-r_j)\;dt-\sum_{\abs{r_j}\le\lambda}
\int_{\R-[-\lambda,\lambda]} \hat h(t-r_j)\;dt +\sum_{\abs{r_j}>\lambda}
\int_{-\lambda}^\lambda \hat h(t-r_j)\;dt.
\]
Using that $\int_\R \hat h(t-r)\;dt=h(0)=1$, we get
\[
2 N^\Gamma(\lambda)=\int_{-\lambda}^\lambda\sum_j \hat h(t-r_j)\;dt+\sum_{\abs{r_j}\le\lambda}
\int_{\R-[-\lambda,\lambda]} \hat h(t-r_j)\;dt-\sum_{\abs{r_j}>\lambda}
\int_{-\lambda}^\lambda \hat h(t-r_j)\;dt.
\]
By (\ref{locbnd}) and (\ref{tmpbnd}) we obtain
\[
N^\Gamma(\lambda)=\frac{\Area(\Gamma\bs\bH)}{4\pi}\lambda^2+O(\lambda\log \lambda).
\]

\section{Preliminaries for the higher rank case} \label{prelim}
\setcounter{equation}{0}
\subsection{}
Fix a positive integer $n$ and let $G=\GL(n)$, considered as an
algebraic group over $\Q$.
Let $P_0$ be the subgroup of upper triangular matrices of $G$, $N_0$
its unipotent radical and $M_0$ the group of diagonal matrices in $G$, which is a Levi subgroup of $P_0$.
Let $A_G$ be the group of scalar matrices with a positive real scalar,
$K_\infty=\mathrm{O}(n)$ and $\Ka=\SO(n)$.
We recall some basic facts about harmonic analysis of $A_G\bs G(\R)/K_\infty$.
For more details cf.~\cite{MR1790156}.
Let $A$ be the group of diagonal matrices with positive real diagonal entries
and determinant $1$. Let
\[
A_G\bs G(\R)=AN_0(\R)K_\infty
\]
be the Iwasawa decomposition.
Let $\vs$ be the Lie algebra of $A$ and $W$ the Weyl group. 
Then $W$ acts on $\vs$ and the Killing form gives a $W$-invariant inner product,
as well as a Haar measure, on $\vs$.
Recall the function $H:A_G\bs G(\R)\to\vs$ defined by
\[
H(\exp(X)nk)=X
\]
for $X\in\vs$, $n\in N_0(\R)$ and $k\in K_\infty$.
Let $\rho\in\vs^*$ be such that $\delta_0(\exp X)^{\frac12}=
e^{\sprod{\rho}{X}}$ for $X\in\vs$ where
$\delta_0$ is the modulus function of $P_0(\R)$.

We will fix Haar measures as in \cite[(3.29)]{MR532745}.
Namely, on $K_\infty$ and $\Ka$ we take the probability measures.
The Haar measure on $\vs$ defined above gives rise to a Haar measure
on $A$ (through the exponential map) and on $\vs^*$ (the dual Haar measure).
We transfer the latter to $\iii\vs^*$ in the obvious way.
On $N_0(\R)$ we take the measure as in [ibid., p.~37] and on $A_G\bs G(\R)$
we take the measure which is compatible with the Iwasawa decomposition.
On quotient spaces (such as $X$) we take the corresponding quotient measures.

Let $\phi_\lambda$, $\lambda\in\vs_{\C}^*$ be the spherical function 
given by Harish-Chandra's formula 
\begin{equation} \label{HC1}
\phi_\lambda(g)=\int_{K_\infty}e^{\sprod{\lambda+\rho}{H(kg)}}\ dk=
\int_{\Ka}e^{\sprod{\lambda+\rho}{H(kg)}}\ dk.
\end{equation}
Let ${\mathcal P}(\vs_{\C}^*)^W$ be the space of $W$-invariant
Paley-Wiener functions on $\vs_{\C}^*$. We will denote by
$C_c^\infty(G(\R)//A_GK_\infty)$ the space of smooth bi-$A_GK_\infty$-invariant
functions which are compactly supported modulo $A_G$.
The Harish-Chandra transform
\[
\HHH:C_c^\infty(G(\R)//A_GK_\infty)\to{\mathcal P}(\vs_{\C}^*)^W
\]
is defined by 
\[
(\HHH f)(\lambda)=\int_{A}f(a)\phi_{\lambda}(a)\ da.
\]
It is the composition of two maps: the Abel transform
\[
\cA:C_c^\infty(G(\R)//A_GK_\infty)\to C_c^\infty(\vs)^W
\]
defined by
\[
\cA(f)(X)=\delta_0(\exp X)^{1/2}\int_{N_0(\R)} f(\exp Xn)\ dn,\quad X\in\vs
\]
and the Fourier-Laplace transform  
\[
\hat h(\lambda)=\int_{\vs}h(X)e^{\sprod\lambda X}\ dX,\ \ \ 
h\in C_c^\infty(\vs),\ \lambda\in\vs_{\C}^*.
\]
Thus $\HHH(f)=\widehat{\cA(f)}$.
The Abel transform is an isomorphism which respects support in the following
sense. Given $R>0$, put 
\[
V(R)=\{X\in\vs:\norm{X}\le R\},\quad U(R)=K_\infty\exp V(R)K_\infty.
\]
Then for all $h\in C_c^\infty(\vs)^W$ with $\supp(h)\subseteq V(R) $ one 
has $\supp(\cA^{-1}(h))\subseteq U(R)$.

Let $\beta(\lambda)=\abs{c(\lambda)c(\rho)^{-1}}^{-2}$ be the Plancherel measure
on $\iii\vs^*$.
The $c$-function is given by the Gindikin-Karpelevic formula
cf.~\cite[p.~46]{MR532745}.
In our case, identifying $\vs^*$ with $\{(\lambda_1,\dots,\lambda_n)\in\R^n:\sum\lambda_i=0\}$ 
we have
\[
c(\lambda)^{-1}=\prod_{1\le i<j\le n}\phi(\lambda_i-\lambda_j)
\]
where we set
\[
\Gamma_{\R}(s)=\pi^{-s/2}{\bf\Gamma}(s/2),\ \ \ \phi(s)=\frac{\Gamma_{\R}(s+1)}{\Gamma_{\R}(s)}.
\]
For any $h\in C_c^\infty(\vs)$ define
\begin{equation}\label{invabel}
\invh h(x)=\frac{1}{\card{W}}\int_{\iii\vs^*}\hat h(\lambda)
\phi_{-\lambda}(x)\beta(\lambda)\ d\lambda.
\end{equation}
By the Plancherel theorem $\invh h\in C_c^\infty(G(\R)//A_GK_\infty)$ and
$\cA(\invh h)=h^W$ where
\[
h^W(a)=\frac1{\card{W}}\sum_{w\in W}h(wa).
\]
By Stirling's formula we have
\[
\phi(z)=O\left(1+\abs{z}^{\frac12}\right),\ \ \ \ 
\frac{\phi'(z)}{\phi(z)}=O\left(\frac 1{1+\abs{z}}\right)
\]
for $\Re z=0$. We infer that
\begin{equation}\label{plnchbnd}
\beta(\lambda)=O(\varbeta(\lambda)),\ \ \lambda\in\iii\vs^*,
\end{equation}
where
\begin{equation}\label{plnchappr}
\varbeta(\lambda)=\prod_{i<j}(1+\abs{\lambda_i-\lambda_j})
\end{equation}
Moreover, for any $\xi\in\vs^*$, denoting by $D_\xi$ the directional derivative
along $\xi$, we have
\begin{equation}\label{logderbnd}
D_\xi\beta(\lambda)=O_\xi\left((1+\norm{\lambda})^{d-r-1}\right),\quad 
\lambda\in\iii\vs^*
\end{equation}
where $r=\dim\vs$ and $d=\dim X$.
In fact, the analogues of these bounds hold for any $G$ -- cf.~\cite[\S3]{MR532745}.
Of course, in our case $r=n-1$ and $d=n(n+1)/2-1$.
It will be convenient to set
\[
\varbeta(t,\lambda)=\prod_{i<j}(t+\abs{\lambda_i-\lambda_j})
\]
so that $\varbeta(\lambda)=\varbeta(1,\lambda)$. Note that
\begin{equation} \label{triv1}
\varbeta(t,\lambda)=O\left((t+\norm{\lambda})^{d-r}\right),\quad\lambda\in\iii\vs^*
\end{equation}
and that
\begin{equation} \label{trivbnd}
\varbeta(t,\lambda_1+\lambda_2)=O\left(\varbeta(t+\norm{\lambda_1},\lambda_2)\right).
\end{equation}

\subsection{}

If $P$ is a parabolic subgroup of $G$ containing $M_0$ then it is defined
over $\Q$ (since $G$ is split) and it admits a unique Levi component $M_P$
(also defined over $\Q$) containing $M_0$.
We call $M_P$ a semi-standard Levi subgroup of $G$ and denote by $\levis$ the set of
semi-standard Levi subgroups of $G$.
Any $M\in\levis$ is isomorphic to $\GL(m_1)\times\dots\times\GL(m_r)$ with $m_1+\dots+m_r=n$.
As a lattice $\levis$ is isomorphic to the partition lattice whose elements
are the partitions of $\{1,\dots,n\}$, ordered by refinement.
We will use superscript $M$ to denote notation pertaining to $M$.
For example, $\vs^M$ is the product of the $\vs$'s corresponding to the $\GL(m_i)$'s
and $\beta^M$ is the Plancherel measure with respect to $M$, etc.
We also set $\vs_M$ to be the orthogonal complement of $\vs^M$ in $\vs$.
Similarly, we write $\vs^*=\vs_M^*\oplus(\vs^M)^*$.
For any $\mu\in\vs^*$ we write $\mu=\mu_M+\mu^M$ corresponding to this
decomposition.

\begin{lemma} \label{scr}
Suppose that $M\ne G$. Then
\[
\varbeta^M(\lambda^M)(1+\norm{\lambda})=O(\varbeta(\lambda)), \ \ \ 
\lambda\in\iii\vs^*.
\]
\end{lemma}

\begin{proof}
Passing to a larger Levi subgroup $M$ only strengthens the inequality.
Therefore, we can assume without loss of generality that $M$ is
maximal. By conjugating $M$ and $\lambda$ by a Weyl element, we can assume
that $M$ is a standard Levi isomorphic to $GL(k)\times GL(n-k)$,
$k\ge 1$. In this case,
\[
\varbeta(\lambda)=\varbeta(\lambda^M)\prod_{(i,j):1\le i\le k<j\le n}(1+\abs{\lambda_i-\lambda_j}).
\]
It remains to observe that
\[
1+\norm{\lambda}=O\Bigl(\prod_{(i,j):1\le i\le k<j\le n}
(1+\abs{\lambda_i-\lambda_j}\Bigr).
\]
In fact,
\[
\norm{\lambda}=O\Bigl(\sum_{(i,j):1\le i\le k<j\le n}
\abs{\lambda_i-\lambda_j}\Bigr),
\]
simply because the linear map $\lambda\mapsto (\lambda_i-\lambda_j)
_{1\le i\le k<j\le n}$ from $\vs^*$ to $\R^{k(n-k)}$ is injective.
\end{proof}

\subsection{}
For each $\lambda\in\vs_{\C}^*$ let $I(\lambda)$
be the unique spherical irreducible subquotient of the induced
representation $\Ind_{P_0(\R)}^{G(\R)}
(e^{\sprod{\lambda}{\Ht(\cdot)}})$. Set
\[
\vs^*_{\untry}=\{\lambda\in\vs_\C^*:I(\lambda)\text{ is unitarizable}\}.
\]
It is known that $\norm{\Re\lambda}\le\norm{\rho}$ for any
$\lambda\in\vs^*_{\untry}$.

Note that for any $M\in\levis$ we have
\[
(\vs^M_{\untry})^*+\iii\vs_M^*\subseteq\vs^*_{\untry}.
\]
For any $w\in W$ let $\vs_{w,\pm1}^*$ be the $\pm$-eigenspaces of $w$
on $\vs^*$. It is well-known that $\vs_{w,+1}^*=\vs_{M_w}^*$ where
$M_w$ is the smallest $M\in\levis$ such that $w$ belongs to the Weyl
group $W_M$ of $M$
(cf.~\cite[p.~1299]{MR681738}, \cite[Theorem 6.27]{MR1217488}).
Define
\[
\vs^*_w=\{\lambda\in\vs^*_{\C}:w\lambda=-\overline{\lambda}\}=
\vs_{w,-1}^*+\iii\vs_{w,+1}^*=\vs_{w,-1}^*+\iii\vs_{M_w}^*.
\]
Since every unitarizable representation is isomorphic to its hermitian
dual, and since $I(\lambda)\simeq I(\lambda')$ if and only if 
$\lambda'$ is in the Weyl orbit of $\lambda$, we infer that
\[
\vs^*_{\untry}\subseteq\cup_{w\in W}\vs^*_w.
\]
For any $M\in\levis$ define
\begin{gather*}
\vs^*_{\hm,\subseteq M}=\cup_{w\in W_M}\vs^*_w,\\
\vs^*_{\untry,\not\subseteq M}=\vs^*_{\untry}\setminus\vs^*_{\hm,\subseteq M}.
\end{gather*}

\subsection{}
Consider now the principal congruence subgroups. Let 
\[
N=\prod_p p^{r_p},\quad r_p\ge 0.
\]
Set
\[
K_p(N)=\{k\in G(\Z_p):k\equiv 1\mod p^{r_p}\Z_p\}
\]
and
\[
K_f(N)=\Pi_{p<\infty}K_p(N).
\]
Then $K_f(N)$ is an open compact subgroup of $G(\A_f)$. The determinant defines
a map which fibers $A_GG(\Q)\bs G(\A)/K_f(N)$ over
\[
\R^+\Q^*\bs\A^*/\prod_p\{r\in\Z_p^*:
a\equiv 1\mod p^{r_p}\Z_p\}\cong (\Z/N\Z)^*.
\]
The fiber of any point is $\SL(n,\R)$-invariant, and as an $\SL(n,\R)$ space
isomorphic to
\[
\SL(n,\Q)\bs\SL(n,\A)/\left(\SL(n,\A)\cap K_f(N)\right).
\]
By strong approximation for $\SL(n)$ we have
$\SL(n,\A)\cap (\SL(n,\R)\cdot K_f(N))=\Gamma(N)$, where $\Gamma(N)
\subseteq\SL(n,\Z)$ is the principal congruence subgroup of level $N$. Thus
each fiber is isomorphic to $\Gamma(N)\bs\SL(n,\R)$, and therefore
\begin{equation}\label{Kf}
A_GG(\Q)\bs G(\A)/K_f(N)\cong\bigsqcup_{(\Z/N\Z)^*}
(\Gamma(N)\bs\SL(n,\R)).
\end{equation}

We fix a Haar measure on $G(\A_f)$. This defines a Haar measure on
$A_G\bs G(\A)$ through $A_G\bs G(\A)=A_G\bs G(\R)G(\A_f)$ and our previous
choice of Haar measure on $A_G\bs G(\R)$.
On any open subgroup of $G(\A_f)$ we take the restricted measure.
Note that the expression $\vol(A_GG(\Q)\bs G(\A)/K_f)$ does not depend on the choice
of Haar measure on $G(\A_f)$.

\section{The method of Duistermaat-Kolk-Varadarajan} \label{DKV}
\setcounter{equation}{0}

Our strategy to prove Theorem \ref{mainthm} is to use the method of 
Duistermaat-Kolk-Varadarajan \cite{MR532745},
in conjunction with  Arthur's (non-invariant) trace formula \cite{MR518111}
which replaces the Selberg trace formula in the non-compact case. 
Arthur's trace formula is an identity of distributions
\begin{equation}\label{tracefor}
J_{\geo}(f)=J_{\spec}(f),\ \ \ f\in C_c^\infty(G(\A)^1).
\end{equation}
The distributions $J_{\geo}$ and $J_{\spec}$ will be recalled in
\S\ref{geom} and \S\ref{spec} respectively.
We apply the trace formula identity to a certain class of test functions
described below.
In order to apply the argument of \cite{MR532745}, the main technical 
difficulty
is to show that, in a suitable sense, the main
terms in the trace formula, in both the geometric and the spectral side,
are exactly those which occur in the compact case.

To make this more precise, fix a compact open subgroup $K_f$ of $G(\A_f)$.
We assume that $K_f\subseteq K_f(N)$ for some $N\ge 3$.
For $h\in C_c^\infty(\vs)$ we define
$\FFF(h)$ to be the restriction to $G(\A)^1$ of the function
$\invh(h)\otimes{\bf 1}_{K_f}$ on $G(\A)$, where $\invh(h)$ is defined by
(\ref{invabel}) and
${\bf 1}_{K_f}$ is the characteristic function of $K_f$ in $G(\A_f)$
normalized by $\vol(K_f)^{-1}$.
For $t\ge1$ let $h_t\in C_c^\infty(\vs)$ be defined by
$h_t(X)=t^rh(tX)$ for $X\in\vs$. We have 
$\widehat{h_t}=\hat h(t^{-1}\cdot)$.
Also, for $\mu\in\iii\vs^*$ we set $h_\mu=he^{-\sprod{\mu}{\cdot}}$.
(Hopefully, this does not create any confusion with the previous notation
$h_t$.)
Finally let $h_{t,\mu}=(h_t)_\mu$ so that $\widehat{h_{t,\mu}}
=\hat h(t^{-1}(\cdot-\mu))$.

Throughout the rest of the paper let $d=\dim X$ and $r=\dim\vs$. For 
$\mu\in\vs^*_\C$ and $t>0$ let
\[
B_t(\mu)=\{v\in \vs^*_\C \colon\norm{v-\mu}\le t\}
\]
be the ball of radius $t$ around $\mu$. If $t=1$ we will often
suppress it from the notation.

Also, for $h\in C_c^\infty(\vs)$ we set 
\[
\M(\hat h)(\lambda)=\max_{\nu\in B_{1+\norm{\rho}}(\lambda)}\abs{\hat h(\nu)}
\quad\mathrm{and}\quad
\nm(h)=\int_{\iii\vs^*}\varbeta(\lambda)\M(\hat h)(\lambda)\ d\lambda,
\]
where $\varbeta(\lambda)$ is defined by (\ref{plnchappr}). We need the 
following auxiliary result.
\begin{lemma} 
For $h\in C_c^\infty(\vs)$ , $t\ge 1$, and $\mu\in\iii\vs^*$ we have
\begin{equation} \label{smp}
\nm(h_{t,\mu})=O_h(t^r\varbeta(t,\mu)).
\end{equation}
\end{lemma}
\begin{proof}
For $t\ge 1$ and $\mu\in\iii\vs^*$ we have 
$\M(\hat h_{t,\mu})(\lambda)\le\M(\hat h)(t^{-1}(\lambda-\mu))$. Hence
\[
\nm(h)\le \int_{\iii\vs^*}\varbeta(\lambda)\M(\hat h)(t^{-1}(\lambda-\mu))\ 
d\lambda =t^r\int_{\iii\vs^*}\varbeta(t\lambda+\mu)\M(\hat h)(\lambda)\ 
d\lambda.
\]
By (\ref{trivbnd}) we get
\[
\varbeta(t\lambda+\mu)=O(\varbeta(t\norm{\lambda},\mu))=
O((1+\norm{\lambda})^{d-r}\varbeta(t,\mu)).
\]
The lemma follows since for any $N\in\N$ we have
$\M(\hat h)(\lambda)=O_{h,N}((1+\norm{\lambda})^{-N})$ for $\lambda\in\iii\vs^*$.
\end{proof}

Let $\Omega\subseteq\iii\vs^*$ be a bounded open subset such that
$\partial\Omega=\overline{\Omega}\setminus\Omega$ is piecewise $C^2$. For $t>0$
let $t\Omega=\{t\mu\colon \mu\in\Omega\}$.

The following Propositions will be proved in \S\ref{geom} and \S\ref{spec} 
respectively.

\begin{prop} \label{maingeom} 
We have
\begin{enumerate} 
\item There exists a neighborhood $\omega_0$ of $0$ in $\vs$ such that
\[
J_{\geo}(\FFF(h))=O(\nm (h)).
\]for all $h\in C_c^\infty(\vs)$ supported in $\omega_0$. 
\item For every such $h$ we have
\[
\int_{t\Omega}J_{\geo}(\FFF(h_\mu))\ d\mu=\vol(G(\Q)\bs G(\A)^1)
\int_{t\Omega}\FFF(h_\mu)(1)\ d\mu + O_{h,\Omega}
\left(t^{d-1}(\log t)^{\max(n,3)}\right)
\]
as $t\to\infty$.
\end{enumerate}
\end{prop}

Analogously,
\begin{prop} \label{mainspec} 
We have
\begin{enumerate}
\item 
$J_{\spec}(\FFF(h))=J_{\di}(\FFF(h))+O(\nm (h))$
for every $h\in C^\infty_c(\vs)$ where 
\begin{equation} \label{Jdi}
J_{\di}(\FFF(h))=\sum_{\pi\in\Pi_{\di}(G(\A))}
\dim\bigl(\HHH_{\pi}^{K}\bigr)\hat h^W(\lambda_{\pi_\infty}).
\end{equation}
\item 
$\int_{t\Omega}J_{\spec}(\FFF(h_\mu))\ d\mu
=\int_{t\Omega}J_{\di}(\FFF(h_\mu))\ d\mu+O_{h,\Omega}\left(t^{d-1}\log t\right)$
as $t\to\infty$.
\end{enumerate}
\end{prop}

Using (\ref{tracefor}) and (\ref{smp}) we infer
\begin{corollary} \label{maintech}
For every $h\in C_c^\infty(\vs)$ supported in $\omega_0$ we have
\begin{enumerate}
\item $J_{\di}(\FFF(h))=O(\nm (h))$.
In particular, for $\mu\in\iii\vs^*$ and $t\ge1$, we have
\[
J_{\di}(\FFF(h_{t,\mu}))=O_h(t^r\varbeta(t,\mu)).
\]
\item
$\int_{t\Omega}J_{\di}(\FFF(h_\mu))\ d\mu=\vol(G(\Q)\bs G(\A)^1)
\int_{t\Omega}\FFF(h_\mu)(1)\ d\mu +
O_{h,\Omega}\left(t^{d-1}(\log t)^{\max(n,3)}\right)$
as $t\rightarrow\infty$.
\end{enumerate}
\end{corollary}

In the rest of this section we will prove the main Theorem using 
Corollary \ref{maintech}.
%Propositions \ref{maingeom} and \ref{mainspec}. 
We first rewrite the expression for $J_{\di}$ as follows.
There exist congruence subgroups $\Gamma_i\subseteq G(\R)$, $i=1,\dots,l,$ such that
\[
A_GG(\Q)\bs G(\A)/K_f\cong\bigsqcup_{i=1}^l(\Gamma_i\bs G(\R)^1).
\]
%Let $\Delta_i$ be the Laplace operator of $\Gamma_i\bs X$. Let 
%$L^2_{\di}(\Gamma_i\bs X)$ be the discrete subspace of $\Delta_i$.
Let $\Lambda_{\di}(\Gamma_i)\subset\vs^*_\C$ be the (symmetrized) spectrum 
of the algebra of invariant differential operators of $\SL(n,\R)$ acting on
the discrete subspace $L^2_{\di}(\Gamma_i\bs X)$ of the Laplace operator on 
$\Gamma_i\bs X$.
For $\lambda\in\Lambda_{\di}(\Gamma_i)$ let 
$m_i(\lambda)$ denote the dimension of the corresponding eigenspace
normalized by the size of the orbit of $\lambda$ under $W$ (cf.~\cite[(3.25)]{MR532745}). 
If $\lambda\notin\Lambda_{\di}(\Gamma_i)$ set $m_i(\lambda)=0$.
Finally, set 
\[
\Lambda_{\di}=\Lambda_{\di}(K_f):=\cup_{i=1}^l\Lambda_{\di}(\Gamma_i),\quad m(\lambda)=
\sum_{i=1}^l m_i(\lambda).
\]
By the well-known relation between automorphic forms and automorphic
representations we get
\[
m(\lambda)=\sum_{\pi:\lambda_{\pi_\infty}=\lambda}\dim\HHH_{\pi}^{K}
\]
for any $\lambda\in\vs^*_{\C}$
and therefore
\begin{equation} \label{newformdis}
J_{\di}(\FFF(h))=\sum_{\lambda\in\Lambda_{\di}(K_f)}m(\lambda)\hat h^W(\lambda).
\end{equation}
For any bounded subset $\Omega\subset\vs^*_\C$ we set 
\[
m(\Omega)=\sum_{\lambda\in\Omega} m(\lambda).
\]
Then
\begin{equation} \label{simp}
\sum_{\substack{\pi\in\Pi_{\di}(G(\A))\\
\lambda_{\pi_\infty}\in t\Omega}}\dim\left(\HHH_\pi^K\right)=m(t\Omega).
\end{equation}

Following \cite{MR532745} we first establish a bound on the 
spectrum near a given point $\mu\in\vs^*$. For this we will only use
the first part of Corollary \ref{maintech}.
%s of Propositions \ref{maingeom} and \ref{mainspec}.

\begin{prop} \label{locuprbnd}
We have
\[
\mult\left(B_t(\mu)\right)=O\left(t^r\varbeta(t,\mu)\right)
\]
for all $\mu\in\iii\vs^*$ and $t\ge1$.
More precisely, for $\mu\in\vs^*$ we have
\[
\mult(\{\nu\in\vs^*_{\C}\colon \norm{\Im\nu-\mu}\le t\})=
O\left(t^r\varbeta(t,\mu)\right).
\]
\end{prop}

\begin{proof}
First note that the second part follows from the first one since $\Re\lambda$
is bounded if $m(\lambda)>0$.
We will prove the first part for $\mu\in\iii\vs_M^*$ by induction on the 
co-rank of $M$.
(The case $M=M_0$ is the statement of the Proposition.)
For the case $M=G$, we have $\mu=0$ and we argue as in 
\cite[Proposition 6.4]{MR532745}.
Fix $h\in C_c^\infty(A)^W$ of small support such that $\hat h\ge0$ on
$\vs^*_{\untry}$, and $\abs{\hat h}\ge1$ on $B(0)$ ([ibid., Lemma 6.2]).
Then $\mult(B_t(0))\le J_{\di}(\FFF(h_t))$, and
$J_{\di}(\FFF(h_t))=O(t^d)$ by Corollary \ref{maintech}.
We deduce the case $M=G$.

For the induction step we proceed as in the proof of [ibid., Proposition
7.1].\footnote{The formulation of Proposition \ref{locuprbnd} is slightly 
stronger
than \cite[Theorem 7.3]{MR532745}. This makes the induction step
a bit smoother.}
Let $h$ be as before. Then $\widehat{h_{t,\mu}}\ge 0$ on $\vs^*_{\hm,\subseteq M}$
and $\abs{\widehat{h_{t,\mu}}}\ge 1$ on $B_t(\mu)$.
Thus,
\[
\begin{split}
\abs{J_{\di}(\FFF(h_{t,\mu}))}=\abs{\sum_{\lambda\in\Lambda_{\di}}m(\lambda)
\widehat{h_{t,\mu}}(\lambda-\mu)}
&\ge\mult(B_t(\mu)\cap\vs^*_{\hm,\subseteq M})\\
&-\sum_{\lambda\in\vs^*_{\untry,\not\subseteq M}}m(\lambda)
\abs{\widehat{h_{t,\mu}}(\lambda-\mu)}.
\end{split}
\]
Hence,
\[
\begin{split} 
\mult(B_t(\mu))&=\mult(B_t(\mu)\cap \vs^*_{\hm,\subseteq M})+
\mult(B_t(\mu)\cap \vs^*_{\untry,\not\subseteq M})\\
&\le\abs{J_{\di}(\FFF(h_{t,\mu}))}
+\sum_{\lambda\in\vs^*_{\untry,\not\subseteq M}}m(\lambda)
\abs{\widehat{h_{t,\mu}}(\lambda-\mu)}
+\mult(B_t(\mu)\cap \vs^*_{\untry,\not\subseteq M}).
\end{split}
\]
We will bound each of the terms on the right-hand side separately.
For the first term we use Corollary \ref{maintech}.
To bound the other two terms consider $\vs^*_w$ with $w\notin W_M$,
or equivalently, $M_w\not\subseteq M$. We have
\[
\vs^*_w\cap\iii\vs_M^*=\iii(\vs_{M_w}^*\cap\vs_M^*)=\iii\vs_{M'}^*
\]
where $M'\in\levis$ is generated by $M$ and $M_w$.
Therefore the kernels of the maps $(\nu,\mu)\mapsto\nu-\mu,\nu-\mu_{M'}$ 
from $\vs^*_w\times\iii\vs_M^*$
to $\vs^*_{\C}$ coincide.
It follows that $\norm{\nu-\mu_{M'}}+\norm{\mu^{M'}}=O(\norm{\nu-\mu})$ for 
$\nu\in\vs^*_w$
and $\mu\in\iii\vs_M^*$. In particular, there exists $c$
such that $B_t(\mu)\cap \vs^*_w=\emptyset$ unless
$\norm{\mu^{M'}}\le ct$, and in this case
\[
B_t(\mu)\cap \vs^*_w\subseteq B_{ct}(\mu_{M'})\cap \vs^*_w.
\]
Therefore, either $\mult(B_t(\mu)\cap \vs^*_w)=0$, or 
\[
\mult(B_t(\mu)\cap \vs^*_w)\le\mult(B_{ct}(\mu_{M'}))=
O(t^r\varbeta(t,\mu_{M'}))=O(t^r\varbeta(t,\mu))
\]
by the induction hypothesis and (\ref{trivbnd}).
Since $\vs^*_{\untry,\not\subseteq M}\subseteq\cup_{w\notin W_M}\vs^*_w$
we obtain
\[
\mult(B_t(\mu)\cap \vs^*_{\untry,\not\subseteq M})=O(t^r\varbeta(t,\mu)).
\]
Therefore for any $N>0$, $\sum_{\lambda\in\vs^*_{\untry,\not\subseteq M}}m(\lambda)
\abs{\hat h(t^{-1}(\lambda-\mu))}$ is bounded by a constant multiple of
\[
\sum_{k=1}^\infty\mult(B_{tk}(\mu)\cap \vs^*_{\untry,\not\subseteq M})k^{-N}=
O\left(\sum_{k=1}^\infty k^{-N+r}t^r\varbeta(kt,\mu)\right).
\]
For $N$ sufficiently large, this is $O(t^r\varbeta(t,\mu))$.
This completes the induction step and the proof of the Proposition.
\end{proof}

\begin{corollary} \label{compl}
$\mult(B_t(0)\setminus\iii\vs^*)=O(t^{d-2})$.
\end{corollary}

Indeed suppose that $\lambda\in\vs^*_{\untry}\cap\vs_w^*$
with $w\ne1$.
Then $\Re\lambda$ is bounded while $\Im\lambda\in\vs_{M_w}^*$ and
$M_w\ne M_0$. Thus,
\[
\mult(B_t(0)\setminus\iii\vs^*)\le\sum_{M_0\neq M\in\levis}
\mult(B_t(0)\cap {\iii\vs_M^*}+C)
\]
for some compact set $C$ of $\vs^*_\C$.
Each of the sets $B_t(0)\cap\iii\vs_M^*+C$, $M\ne M_0$, is covered by 
$O(t^{r-1})$ balls of bounded radius whose centers are in 
$B_t(0)\cap\iii\vs_M^*$.
The corollary follows from Proposition \ref{locuprbnd}
since $\varbeta(R,\mu)=O_R(\norm{\mu}^{d-r-1})$ for $\mu\in\iii\vs_M^*$.

We are now ready to prove Theorem \ref{mainthm}.
Let $\Omega\subseteq\iii\vs^*$ be a bounded domain with piecewise $C^2$
boundary. Fix $h\in C_c^\infty(\vs)^W$ with $h(0)=1$.
Using (\ref{newformdis}) and the second part of Corollary \ref{maintech} we get
\[
\sum_{\lambda\in\Lambda_{\di}}m(\lambda)\int_{t\Omega}
\hat h(\lambda-\mu)\ d\mu
=\vol(G(\Q)\bs G(\A)^1)
\int_{t\Omega}\FFF(h_\mu)(1)\ d\mu
+O\left(t^{d-1}(\log t)^{\max(n,3)}\right).
\]
as $t\rightarrow\infty$. By Plancherel inversion
\[
\int_{t\Omega}\FFF(h_\mu)(1)\ d\mu=\frac{1}{\card{W}\vol(K_f)}
\int_{t\Omega}\int_{\iii\vs^*}
\hat h(\lambda-\mu)\beta(\lambda)\,d\lambda\,d\mu.
\]
Note that $\vol(G(\Q)\bs G(\A)^1)\vol(K_f)^{-1}=\vol(G(\Q)\bs G(\A)^1/K_f)$.
As in \cite[\S8]{MR532745} we will show that
\begin{gather}
\label{error1} \int_{t\Omega}\int_{\iii\vs^*}
\hat h(\lambda-\mu)\beta(\lambda)\,d\lambda\,d\mu
-\int_{t\Omega}\beta(\lambda)\,d\lambda=O(t^{d-1}),\\
\label{error2} \sum_{\lambda\in\Lambda_{\di}}m(\lambda)\int_{t\Omega}
\hat h(\lambda-\mu)\ d\mu-\mult(t\Omega)=O(t^{d-1}).
\end{gather}
From the description of the residual spectrum (\cite{MR1026752}) it follows
that $\Lambda_{\di}\setminus\Lambda_{\cu}\subseteq
\vs^*_{\C}\setminus\iii\vs^*$. Therefore, in view of Corollary \ref{compl}
we can replace $\Lambda_{\di}$ by $\Lambda_{\cu}$ in (\ref{error2}).
Altogether, this gives Theorem \ref{mainthm} by (\ref{simp}).

To prove (\ref{error1}) we can write its left-hand side as
\begin{equation}\label{error3}
\int_{t\Omega'}\int_{t\Omega}\hat h(\lambda-\mu)\beta(\lambda)\,d\mu\ d\lambda
-\int_{t\Omega}\int_{t\Omega'}\hat h(\lambda-\mu)\beta(\lambda)\,d\mu\ d\lambda
\end{equation}
where $\Omega'$ is the complement of $\Omega$ in $\iii\vs^*$
(cf.~\cite[p.~84]{MR532745}). For $\kappa\ge 0$ let 
\[
\partial_\kappa(t\Omega)=
\big\{\nu\in\iii\vs^*:\inf_{\mu\in t\partial\Omega}\norm{\nu-\mu}\le\kappa\big\}.
\]
By separating the integrals in (\ref{error3}) into shells 
$k-1\le\norm{\lambda-\mu}<k$, $k\in\N$, we can bound the sum by
\[
\sum_{k=1}^\infty
\int_{\partial_k(t\Omega)}\int_{(B_k(\lambda)\setminus
B_{k-1}(\lambda))\cap\iii\vs^*}
\abs{\hat h(\lambda-\mu)}\beta(\lambda)\,d\mu\ d\lambda.
\]
Using (\ref{plnchbnd}) and the rapid decay of $\hat h$,
for any $N>0$ this is bounded by a constant multiple of
\[
\sum_{k=1}^\infty k^{-N}k^r\vol(\partial_k(t\Omega))(k+t)^{d-r}.
\]
Since $\Omega$ has a piecewise $C^2$ boundary we can cover $\partial_k(t\Omega)$
by $O(t^{r-1})$ balls of radius $k$ and (\ref{error1}) follows.
Similarly, we write the left-hand side of (\ref{error2}) as
\begin{equation}\label{error4}
\begin{split}
-\sum_{\lambda\in\Lambda_{\di},\lambda\in t\Omega}m(\lambda)
\int_{t\Omega'}\hat h(\lambda-\mu)\,d\mu
&+\sum_{\lambda\in\Lambda_{\di},\lambda\in t\Omega'}m(\lambda)
\int_{t\Omega}\hat h(\lambda-\mu)\,d\mu\\
&+\sum_{\lambda\in\Lambda_{\di},\lambda\notin\iii\vs^*}m(\lambda)
\int_{t\Omega}\hat h(\lambda-\mu)\,d\mu.
\end{split}
\end{equation}
We proceed as above and divide the domains $t\Omega\times t\Omega'$,
$t\Omega'\times t\Omega$ and $(\vs_\C^*\setminus \iii\vs^*)\times
t\Omega$ into
shells $k-1<\norm{\lambda-\mu}<k$, $k\in\N$, and estimate the 
corresponding sum-integral on each shell. In this way we bound 
(\ref{error4})  by
\[
\sum_{k=1}^\infty k^{-N}(\mult(\partial_k(t\Omega))+
\mult(B_{R(k+t)}(0)\setminus\iii\vs^*))k^r
\]
where $R$ is such that $\Omega\subset B_R(0)$.
We can now use the condition on $\partial\Omega$,
Proposition \ref{locuprbnd} and Corollary \ref{compl}
to infer (\ref{error2}).

Next we show that Theorem \ref{mainthm} implies Corollary \ref{maincor}.
Let $\Omega\subset\iii\vs^*$ be as in Theorem \ref{mainthm}.
If $K_f$ equals the principal congruence subgroup $K_f(N)$ then it follows
from (\ref{Kf}) that
\[
\sum_{\substack{\pi\in\Pi_{\cu}(G(\A))\\
\lambda_{\pi_\infty}\in t\Omega}}\dim\left(\HHH_\pi^K\right)=
\varphi(N)\sum_{\lambda\in\Lambda_{\cu}(\Gamma(N)),\lambda\in t\Omega}m(\lambda).
\]
and
\[
\vol(G(\Q)\bs G(\A)^1/K_f(N))=\varphi(N)\vol(\Gamma(N)\bs X)
\]
where $\varphi(N)=\#[(\Z/N\Z)^*]$.
Thus, Theorem \ref{mainthm} in this case amounts to Corollary \ref{maincor}.

Corollary \ref{maincor2} is derived from Corollary \ref{maincor}
exactly as in \cite[p.~87-88]{MR532745}.
Note that the main term in Corollary \ref{maincor} differs from that of
\cite[Theorem 8.8]{MR532745} by a certain factor $\sigma(G)$ ([ibid., (8.12)]).
This is because of the different normalization of measures on $X$ and $\iii\vs^*$.
This has no effect on Corollary \ref{maincor2} where (as in [ibid., (8.33)]) it is assumed that
both the volume and
Laplacian are with respect to the same Riemannian structure.

\begin{remark} \label{rem10}
All the Propositions in this section carry over, in the obvious way,
to groups which are products of $\GL(m)$'s.
The proof of Proposition \ref{locuprbnd} depends on the validity
of Propositions \ref{maingeom} and \ref{mainspec}.
Proposition \ref{maingeom} will be proved in the next section.
We will prove Proposition \ref{mainspec} in \S\ref{spec} below by induction on $n$. The induction hypothesis, together with the validity
of Proposition \ref{maingeom} guarantees the validity of Proposition \ref{locuprbnd}
for any proper Levi subgroup of $G$.
\end{remark}

\section{The geometric side} \label{geom}
\setcounter{equation}{0}

In this section we study the geometric side $J_{\geo}$ of the trace formula and
prove Proposition \ref{maingeom}.
By the coarse geometric expansion we can write $J_{\geo}$ as a sum of 
distributions
\[
J_{\geo}(f)=\sum_{\ho\in\cO}J_{\ho}(f),\quad f\in C^\infty_c(G(\A)^1),
\]
parameterized by semisimple conjugacy classes of $G(\Q)$.
The distribution $J_{\ho}(f)$ is the value at $T=0$
of the polynomials $J_{\ho}^T(f)$ defined in \cite{MR518111}.
In particular, following Arthur, we write $J_{\unip}(f)$ for the 
contribution corresponding to the class of $\{1\}$.

The first step is to show that it suffices to deal with $J_{\unip}(f)$.

\begin{lemma}\label{unipfree}
Let $K_f(N)$ be a principal congruence subgroup of level $N\ge3$.
There exists a bi-$K_\infty$-invariant compact neighborhood $\omega$ of $K_f(N)$ 
in $G(\A)^1$ which does not contain any $x\in G(\A)$ whose semisimple part is
conjugate to a non-unipotent element of $G(\Q)$.
\end{lemma}

\begin{proof}
Taking the coefficients of the characteristic polynomial gives rise to a conjugation
invariant algebraic map from $G$ to the affine $n$-space,
and therefore to a continuous map $q:G(\A)\to\A^n$.
Let $T=q(K_\infty\cdot K_f)$.
The discreteness of $\Q^n$ implies that there exists a neighborhood $\omega$
of $K_\infty\cdot K_f$ such that 
\[
q(\omega)\cap\Q^n\subseteq T\cap\Q^n.
\]
By passing to a smaller neighborhood we can assume that $\omega$ is bi-$K_\infty$-invariant and compact.

To prove the Lemma it suffices to show that
\begin{equation} \label{TQ}
T\cap\Q^n=\{q(1)\}.
\end{equation}
Indeed, if the semi-simple part of $x\in\omega$ is conjugate to 
$\gamma\in G(\Q)$
then $q(x)=q(\gamma)\in q(\omega)\cap\Q^n\subseteq T\cap\Q^n$ and therefore 
by (\ref{TQ}) $q(\gamma)=q(1)$
so that $\gamma$ is unipotent.

To show (\ref{TQ}), let $f$ be a monic polynomial over $\Q$ with coefficients 
in $T$.
Since $T_p\subseteq\Z_p^n$ for all $p$, we infer that $f$ has integral 
coefficients,
so that the roots of $f$ are algebraic integers. Moreover, by the condition
at $\infty$, the roots of $f$ have absolute value $1$.
It follows from Dirichlet's unit theorem, that the roots of $f$ are roots of 
unity.
Let $g\in K_p(N)$. Then $g-1\in NM_n(\Z_p)$ and therefore the roots of the
characteristic polynomial of $g$ are congruent to $1$ modulo $N$ in the ring
of integers of the algebraic closure of $\Q_p$.
Thus, the roots of $f$ are roots of unity which are congruent to $1$ modulo $N$
in the ring of integers of the corresponding cyclotomic field.
By Serre's lemma (\cite[p.~207]{MR0282985}) they are all $1$, and (\ref{TQ}) 
follows.
\end{proof}

\begin{corollary}\label{c3.2}
Let $\omega$ be as in Lemma \ref{unipfree} and suppose that 
$f\in C_c^\infty(G(\A)^1)$ is supported 
in $\omega$. Then 
\[
J_{\geo}(f)=J_{\unip}(f).
\]
\end{corollary}
\begin{proof} Let $\ho\in\cO$. By \cite[Theorem 8.1]{MR518111} the distribution
$J^T_\ho(f)$ is given by
\[
J^T_\ho(f)=\int_{G(\Q)\bs G(\A)^1}j^T_\ho(x,f)\ dx,
\]
where $j^T_\ho(x,f)$ is defined in [ibid., p. 947].
It follows from Lemma \ref{unipfree} and the definition of $j^T_\ho(x,f)$
that $J^T_\ho(f)=0$ unless $\ho$ corresponds to $\{1\}$. 
\end{proof}

To analyze $J_{unip}(f)$ we use Arthur's fundamental result 
(\cite[Corollaries 8.3 and 8.5]{MR828844}) to
express $J_{\unip}(f)$ in terms of weighted orbital integrals. To state the 
result we recall some facts about weighted orbital integrals. Let $S$ be a 
finite set of places of $\Q$ containing $\infty$. Set
\[
\Q_S=\prod_{v\in S}\Q_v,\quad \mathrm{and}\quad G(\Q_S)=\prod_{v\in S}G(\Q_v).
\]
Let $M\in\levis$ and $\gamma\in M(\Q_S)$. 
The general weighted orbital integrals $J_M(\gamma,f)$ defined in 
(\cite{MR932848}) are distributions on $G(\Q_S)$. Let
 \[
G(\Q_S)^1=G(\Q_S)\cap G(\A)^1
\]
and write $C^\infty_c(G(\Q_S)^1)$ for the space of functions on $G(\Q_S)^1$
obtained by restriction of functions in $C^\infty_c(G(\Q_S))$. If $\gamma$ 
belongs to the intersection of $M(\Q_S)$ with $G(\Q_S)^1$, one can obviously
define the corresponding weighted orbital integral as linear form on 
$C^\infty_c(G(\Q_S)^1)$. 
Since for $\GL(n)$ all conjugacy classes are stable, 
the expression of $J_{\unip}(f)$ in terms of weighted 
orbital integrals  simplifies. For 
$M\in\levis$ let $\left(\cU_M(\Q)\right)$ be the (finite) set of unipotent 
conjugacy classes of $M(\Q)$. Then by 
(\cite[Corollaries 8.3 and 8.5]{MR828844}) we get
\begin{equation}\label{3.8}
J_{\unip}(f)=\vol(G(\Q)\bs G(\A)^1)f(1)+\sum_{\substack{M\in\levis\\M\not=G}}
\sum_{u\in\left(\cU_M(\Q)\right)}a^M(u)J_M(u,f)
\end{equation}
for $f\in C^\infty_c(G(\Q_S)^1)$, where  $a^M(u)$ are certain constants which 
we fortunately do not need to worry about for the purpose of this paper. 

Thus, by (\ref{plnchbnd}), in order to prove Proposition \ref{maingeom} it suffices to show that 
for every $M$, $u\in\cU_M(\Q)$ and $h\in C^\infty_c(\vs)$ with $\supp h\subset
\omega_0$, 
\begin{equation} \label{pntineq}
J_M(u,\FFF(h))=O\left(\int_{\iii\vs^*}\abs{\hat h(\lambda)}\beta(\lambda)
\ d\lambda\right),
\end{equation}
and if $\Omega$ is bounded with a piecewise $C^2$ boundary then 
for all $M\ne G$, $u\in\cU_M(\Q)$,
\begin{equation} \label{integineq}
\int_{t\Omega}J_M(u,\FFF(h_\mu))\ d\mu=O\left(t^{d-1}(\log t)^{\max(r+1,3)}\right)
\end{equation}
as $t\to\infty$.

Before studying the distributions $J_M(u,f)$ we first examine a special
case of unipotent orbital integrals.
Recall that any parabolic subgroup $P$ of (any reductive group) $G$ has an
open orbit on its unipotent radical $N_P$ (acting by conjugation).
The corresponding conjugacy class $\orbit_P=\orbit_P^G$ in $G$ is a unipotent
orbit which is called a Richardson orbit.
We have $\dim\orbit_P=2\dim N_P$ (\cite[Theorem 7.1.1]{MR1251060}). % (cf.~Bala-Carter?).
For $G=\GL(n)$, all orbits are Richardson ([ibid., Theorem 7.2.3])
and $G(F)$ (resp.~$P(F)$) acts transitively on $\orbit_P(F)$
(resp.~$N_P(F)\cap\orbit_P(F)$) over any field $F$.

The following Lemma is probably well-known but
we were unable to locate a convenient reference in the literature.

\begin{lemma} \label{Richardson}
Let $P=MN$ be a parabolic subgroup of $G=\GL(n)$.
Then the orbital integral of $f\in\C_c^\infty(G(\Q_S))$ along $\orbit_P(\Q_S)$
is given by
\[
\int_{K_S}\int_{N(\Q_S)}f(k^{-1}uk)\ du\ dk
\]
where $K_S$ is the maximal compact subgroup of $G(\Q_S)$.
\end{lemma}

\begin{proof}
Fix a representative $u\in\orbit_P(\Q_S)\cap N(\Q_S)$. Then, denoting
by $C_g^H$ the centralizer of $g$ in $H$, we have
\[
\dim C^G_u=\dim G-\dim\orbit_P=\dim G-2\dim N=\dim P-\dim N=\dim C^P_u.
\]
We can therefore write the orbital integral as
\[
\int_{K_S}\int_{C^P_u(\Q_S)\bs P(\Q_S)}f(k^{-1}p^{-1}upk)\delta_P(p)^{-1}\ dp\ dk.
\]
The map $p\mapsto p^{-1}up$ is a submersion from $P(\Q_S)$ to $N(\Q_S)$. 
Therefore the functional
\[
g\in C_c^\infty(N(\Q_S))\mapsto\int_{C^P_u(\Q_S)\bs P(\Q_S)}g(p^{-1}up)\delta(p^{-1})\ dp
\]
is absolutely continuous with respect to the Haar measure of $N(\Q_S)$.
Since the orbit of $u$ under $P(\Q_S)$ is dense in $N(\Q_S)$,
the only absolutely continuous measure $\mu$ on $N(\Q_S)$ satisfying
$\mu(Ad(p)f)=\delta(p)\mu(f)$ is the Haar measure.
The Lemma follows.
\end{proof}

We fix an open compact subgroup $K_f\subseteq G(\A_f)$ with
$K_f\subset K(N)$, $N\ge 3$ and put
\[
K=K_\infty\cdot K_f.
\]
Let $\omega$ be as in Lemma \ref{unipfree}.
Note that there exists a neighborhood $\omega_0$ of $0$ in $\vs$ such that
if $h\in C_c^\infty(\vs)^W$ is supported in $\omega_0$ then
$f=\FFF(h)$ is supported in $\omega$.
Let $S$ be a sufficiently large set of places containing $\infty$ and the primes dividing $N$.

The distributions $J_M(\cdot,f)$ are based on the usual orbital integrals
studied by Harish-Chandra and Ranga-Rao.
Unlike in the case $n=2$, their Fourier transforms are not explicit
in general, so we cannot argue as in \S\ref{SL2} (cf.~(\ref{logamma})).
Instead, following Arthur, we write $J_M(\cdot,f)$ in terms of weighted orbital
integrals, namely as the expression on top of p.~256 in \cite{MR932848}
\footnote{Note the following typo: $v_M^Q$ should be replaced by $w_M^Q$.}.
In the case where $\gamma$ is unipotent, that is, $\sigma=1$ in the notation of [ibid.],
the formula simplifies since the outer integration disappears and only $Q=G$ contributes.
(Note that the function $v'_Q$ vanishes on $K$ unless $Q=G$.)
The derivation of \cite[p.~256]{MR932848} is based on an expression of
the invariant orbital integral in $M(\Q_S)$ and in particular
on the formula on [ibid., p.~246]).
In our case all unipotent orbits of $M$ are stable and are of Richardson type.
Therefore, any unipotent $u$ of $M(\Q)$ belongs to $\orbit_{Q\cap M}^M$
for some parabolic $Q$ of $G$ and we can use Lemma \ref{Richardson} instead of [loc.~cit.] to simplify the expression for $J_M(u,f)$ to
\[
J_M(u,f)=\int_K\int_Uf(k^{-1}xk)w_M(x)\ dx
\]
where $U$ is an appropriate compact subset of $N_Q(\Q_S)$
(depending on the support of $f$)
and the weight function $w_M(x)$ is described in [ibid., Lemma 5.4]:
it is a finite linear combination
of functions of the form $\prod_{i=1}^r\log{\norm{p_i(x_{v_i})}_{v_i}}$
where $p_i$ are polynomials on $N$ into an affine space,
$v_i\in S$, $i=1,\dots,r$ (not necessarily distinct) and
\[
\norm{(y_1,\dots,y_m)}_v=
\begin{cases}\max(\abs{y_1}_v,\dots,\abs{y_m}_v),&v<\infty,\\
\abs{y_1}_v^2+\dots+\abs{y_m}_v^2,&v=\infty.\end{cases}
\]
(The fact that the product is over $r$ terms is implicit in [loc.~cit.]
but follows from the proof.)
Using the Plancherel theorem, and viewing $\phi_{\lambda}$
as a function on $N_Q(\Q_S)$ depending only on the Archimedean component,
we have
\[
J_M(u,\FFF(h))=\int_{\iii\vs^*}\hat h(\lambda)\beta(\lambda)
\left(\int_Uw_M(x)\psi(x)\phi_{-\lambda}(x)\ dx\right)\ d\lambda
\]
where $\psi(x)=\int_{K_S}{\bf 1}_{K_f}(k^{-1}xk)\ dk$.
Since $\abs{\phi_\lambda}\le 1$, the inner integral is majorized by
\[
\int_U\abs{w_M(x)}\ dx,
\]
which converges by \cite[Lemma 6.1]{MR932848},
and (\ref{pntineq}) follows.

Now, by a change of variable we can write
\[
J_M(u,\FFF(h_\mu))=\int_{\iii\vs^*}\hat h(\lambda)\beta(\lambda+\mu)
\left(\int_Uw_M(x)\psi(x)\phi_{-\lambda-\mu}(x)\ dx\right)\ d\lambda.
\]
It follows from [ibid., Lemma 7.1] that for some $a>0$ we have
\[
\int_{\{n\in U:\abs{w_M(x)}>aR\}}\abs{w_M(x)}\ dx=O\left(e^{-R^{1/r}}\right).
\]
Taking $R=(\log t)^r$ we get
\[
\int_Uw_M(x)\psi(x)\phi_{-\lambda-\mu}(x)\ dx=
\int_{U_t}w_M(x)\psi(x)\phi_{-\lambda-\mu}(x)\ dx+O\left(1/t\right)
\]
where we set $U_t=\{x\in U:\abs{w_M(x)}<a(\log t)^r\}$.
Therefore, using (\ref{plnchbnd}) and (\ref{triv1}) we obtain
\[
J_M(u,\FFF(h_\mu))=
\int_{\iii\vs^*}\hat h(\lambda)\beta(\lambda+\mu)
\left(\int_{U_t}w_M(x)\psi(x)\phi_{-\lambda-\mu}(x)\ dx\right)\ d\lambda+
O\left((1+\norm{\mu}^{d-r})/t\right).
\]
In order to prove (\ref{integineq}), we use Harish-Chandra's formula for 
the spherical function (\ref{HC1}) to write $J_M(u,\FFF(h_\mu))$,
up to an error term of $O\left((1+\norm{\mu}^{d-r})/t\right)$ as
\[
\int_{\iii\vs^*}\hat h(\lambda)\int_{\Ka}\int_{U_t}e^{\sprod{-\lambda+\rho}{H(kx)}}
w_M(x)\psi(x)\beta(\lambda+\mu)e^{-\sprod{\mu}{H(kx)}}\ dx\ dk\ d\lambda.
\]
Integrating over $t\Omega$ and interchanging the order of integration, we 
obtain that
up to an error of order $O(t^{d-r-1})$, the left-hand side of (\ref{integineq}) 
is equal to
\[
t^r\int_{\iii\vs^*}\hat h(\lambda)
\int_{\Ka}\int_{U_t}e^{\sprod{-\lambda+\rho}{H(kx)}}
w_M(x)\psi(x)\int_{\Omega}\beta(\lambda+t\mu)e^{-t\sprod{\mu}{H(kx)}}
\ d\mu\ dx\ dk\ d\lambda.
\]

We first estimate the integral over $\Omega$.
The trivial estimate using (\ref{plnchbnd}) and (\ref{triv1}) gives
$O((\norm{\lambda}+t)^{d-r})$.
To go further, fix a regular $\xi\in\vs^*$ and 
let $\nu$ be the outward pointing normal vector
field at the boundary $\partial\Omega$. Let $k\in\Ka$ and $x\in N(\R)$ be
such that $\sprod{\xi}{H(kx)}\not=0$. 
Then by the divergence theorem, we get
\begin{multline*}
\int_{\Omega}\beta(\lambda+t\mu)e^{-t\sprod{\mu}{H(kx)}}\ d\mu =
\frac{-1}{t\sprod{\xi}{H(kx)}}\int_{\Omega}\beta(\lambda+t\mu)
D_\xi e^{-t\sprod{\mu}{H(kx)}}\ d\mu=\\
\frac1{\sprod{\xi}{H(kx)}}
\int_{\Omega}D_\xi\beta(\lambda+t\mu)e^{-t\sprod{\mu}{H(kx)}}\;d\mu
-\frac1{t\sprod{\xi}{H(kx)}}
\int_{\partial\Omega}\beta(\lambda+t\sigma)
e^{-t\sprod{\sigma}{H(kx)}}\sprod{\nu(\sigma)}{\xi}\;d\sigma.
\end{multline*}
We estimate the integrands on the right-hand side by (\ref{logderbnd})
and (\ref{plnchbnd}) respectively. We obtain
\begin{gather*}
\int_{\Omega}(D_\xi\beta)(\lambda+t\mu)e^{-t\sprod{\mu}{H(kx)}}\ d\mu=O\left(1+\norm{\lambda}^{d-r-1}+t^{d-r-1}\right),\\
\int_{\partial\Omega}\beta(\lambda+t\sigma)
e^{-t\sprod{\sigma}{H(kx)}}\sprod{\nu(\sigma)}{\xi}\ d\sigma=
O\left(1+\norm{\lambda}^{d-r}+t^{d-r}\right).
\end{gather*}
Thus,
\[
\int_{\Omega}\beta(\lambda+t\mu)e^{-t\sprod{\mu}{H(kx)}}\ d\mu= 
O\left(\frac{1+\norm{\lambda}^{d-r}+t^{d-r-1}}{\abs{\sprod{\xi}{H(kx)}}}\right).
\]

Set $F(k,x)=\sprod{\xi}{H(kx)}$ on $\Ka\times U$ and for any $\eps>0$ let
$V_{<\eps}=\{(k,x)\in\Ka\times U:\abs{F(k,x)}<\eps\}$.
Similarly for $V_{\ge\eps}$.

It follows that the left-hand side of (\ref{integineq}) is majorized by
\begin{equation} \label{finMAJ}
O\left(t^d(\log t)^r\right)
\int_{\vs^*}\abs{\hat h(\lambda)}(1+\norm{\lambda})^{d-r}\ d\lambda\ 
\left(\vol(V_{<\eps})+\int_{V_{\ge\eps}}\frac{dy}{t\abs{F(y)}}\right)
\end{equation}
uniformly for $\eps>0$.

Next, we study the critical points of $F$ on $\Ka\times N_Q(\R)$.
Let $K_L=\Ka\cap Q(\R)$. Note that $K_L$ acts on $\Ka\times N_Q(\R)$ by the free action
$m(k,x)=(km^{-1},mxm^{-1})$, 
and that $F$ is invariant under this action. The following Lemma
(and its proof) hold for any real reductive group.

\begin{lemma} \label{critpts}
$F$ is a Morse function on $K_L\bs(\Ka\times N_Q(\R))$.
Its critical points are the cosets $K_L(w,1)$, $w\in W$.
\end{lemma}

\begin{proof}
Let $\gf_1$ be the Lie Algebra of $G_1(\R)=\SL(n,\R)$, and let $\gf_1=\kf\oplus
\pg$ be the Cartan decomposition. 
We will write $\sprod{\cdot}{\cdot}$ for the Killing form of $\gf_1$ and use it
to identify $\vs^*$ with $\vs$.
Suppose that $(k,x)$ is a critical point. In particular, $k$ is a critical
point for the function $F(\cdot,x)$. 
By \cite[Lemma 5.3]{MR707179} this means that $kx\in A\Ka$.
Let $\nnn_Q$ denote the Lie algebra of $N_Q(\R)$ and we use the notation
$F_X$, $X\in\kkk\oplus\nnn_Q$ for the $X$-directional derivative. Then for
any $Y\in\nnn_Q$
\[
0=F_Y(k,x)=\frac{d}{dt}\sprod{\xi}{\Ht(k\exp(tY)x)}\big|_{t=0}=
\sprod{\xi}{\Ht(\cdot)}_{\Ad(k)Y}(kx),
\]
where the latter denotes the directional derivative along $\Ad(k)Y$ of the
function $\sprod{\xi}{\Ht(\cdot)}$.
By \cite[Corollary 5.2]{MR707179} and the fact that $kx\in A\Ka$ we get
\[
0=\sprod{\xi}{\Ad(k)Y}.
\]
Thus, $X:=\Ad(k^{-1})\xi\in\nnn_Q^\perp\cap\ppp=\mf_Q\cap\ppp$.
Therefore, there exists $k'\in K_L$ such that $\xi'=\Ad(k')X\in\vs$.
In particular, $\xi'\in\proj_\vs(\Ad(\Ka)^{-1}\xi)$
and $\norm{\xi'}=\norm{\xi}$. On the other hand, by Kostant's convexity theorem
\cite[p.~473]{MR1790156} 
it follows that $\proj_\vs(\Ad(\Ka)^{-1}\xi)$ equals to the convex hull of 
$w\xi$, $w\in W$.
We infer that $\xi'=w^{-1}\xi$ for some $w$. Therefore,
$k\in C_K(\xi)wk'$. Since $\xi$ is regular, $C_K(\xi)=C_K(A)\subseteq K_L$
and thus $k\in WK_L$. Write $kx=ak_1$ and $k=wk_2$ 
with $a\in A$ and $k_1\in\Ka$ and $k_2\in K_L$. The equality
$ak_1=wk_2x$ gives $w^{-1}k_1=a'x'k_2$ where $a'=w^{-1}a^{-1}w\in A$ and
$x'=k_2xk_2^{-1}\in N_Q(\R)$. By the uniqueness of the Iwasawa decomposition
we obtain $x'=1$ and therefore $n=1$.

Consider the Hessian as a symmetric bilinear form on the tangent space
$\kkk_L\bs\kkk\oplus\nnn_Q$.
Note that the Killing from defines a perfect pairing on $\nnn_Q\times\kkk_L\bs\kkk$.
Since $F(wk,1)=1$ for all $k$, $\kkk_L\bs\kkk$ is totally isotropic
for the Hessian.
For $Y\in\nnn_Q$ we get as above 
\[
F_Y(wk,1)=\frac d{dt}\sprod{\xi}{H(wke^{tY})}|_{t=0}=
\frac d{dt}\sprod{\xi}{H(e^{t\Ad(wk)Y})}|_{t=0}=
\sprod{\xi}{\Ad(wk)Y}.
\]
Thus, for $X\in\kkk_L\bs\kkk$ we get
\[
F_{XY}(w,1)=\sprod{\xi}{\Ad(w)\ad(X)Y}=\sprod{\Ad(w)^{-1}\xi}{\ad(X)Y}=
\sprod{\ad(\Ad(w)^{-1}\xi)Y}X.
\]
Hence, the Hessian is not singular since $\ad(\Ad(w)^{-1}\xi)$ is non-singular on $\nnn_Q$.
\end{proof}

The estimation (\ref{integineq}) will follow from (\ref{finMAJ}) by taking
$\eps=1/t$ and using the following Lemma applied to $F$.
(Note that $\dim\nnn_Q>1$ unless $n=2$.)

\begin{lemma}\label{morse}
Let $X$ be an orientable manifold of dimension $m\ge 2$ with a
nowhere vanishing differential $m$-form, and let $dx$ be the corresponding measure.
Let $f$ be a Morse function on $X$. Then for any compact $Y\subseteq X$
\begin{gather*}
\vol\{x\in Y:\abs{f(x)}<\delta\}=O\left(\delta(-\log\delta)^{\eta}\right)\\
\int_{\{x\in Y:\abs{f(x)}\ge\delta\}}\frac1{\abs{f(x)}}\ dx=
O\left((-\log\delta)^{1+\eta}\right)
\end{gather*}
for $\delta<\frac12$, where $\eta=1$ if $m=2$ and $\eta=0$ otherwise.
\end{lemma}

\begin{proof}
The question is local. Therefore, by the Morse Lemma it suffices to consider 
the following 3 cases.
\begin{enumerate}
\item $f\ne0$ on $Y$.
\item $X=\R^m$ and $f(\underline{x})=x_1$.
\item $X=\R^m$, $m=p+q$ and
\[
f(\underline{x},\underline{y})=\abs{\underline{x}}^2-\abs{\underline{y}}^2
\ \ \ \underline{x}\in\R^p,\ \underline{y}\in\R^q.
\]
\end{enumerate}
The first case is trivial.
In the second case the volume is bounded by $\delta$ while the
integral is bounded by $-\log\delta$.
Consider the last case. If $pq=0$ the volume is the volume of a ball of
radius $\sqrt{\delta}$ which is $O(\delta^{m/2})$.
The integral reduces to
\[
\int_{\underline{x}\in\R^m:
\delta\le\abs{\underline{x}}^2\le 1}\frac1{\abs{\underline{x}}^2}
\ d\underline{x}=O(\left(\int_{\delta^2}^1r^{m-1}r^{-2}\ dr\right)=
O(-\log\delta).
\]
For $pq>0$ the volume of the intersection of the unit disc with
\[
\{(\underline{x},\underline{y}):\underline{x}\in\R^p,\underline{y}\in\R^q,
\abs{\abs{\underline{x}}^2-\abs{\underline{y}}^2}<\delta\}
\]
is given by a constant multiple of
\[
\int_{0\le r,s\le1:\abs{r^2-s^2}<\delta}r^{p-1}s^{q-1}\ dr\ ds.
\]
If $p=q=1$ then this is majorized by
\[
\int_{0\le x,y\le 1:xy<\delta}\ dx\ dy=
\delta-\delta\log\delta.
\]
Otherwise, it is majorized by
\[
\int_{0\le x,y\le 1:xy<\delta}x+y\ dx\ dy\le\delta
\]
Similarly, the integral reduces to
\[
\int_{\substack{\abs{\underline{x}},\abs{\underline{y}}\le1\\
\abs{\abs{\underline{x}}^2-\abs{\underline{y}}^2}\ge\delta}}
\frac1{\abs{\abs{\underline{x}}^2-\abs{\underline{y}}^2}}\ d\underline{x}
\ d\underline{y}.
\]
Using polar coordinates for $\underline{x}$, $\underline{y}$, we obtain
\[
\int_{\substack{0\le r,s\le 1\\\abs{r^2-s^2}\ge\delta}}
\frac{r^{p-1}s^{q-1}}{\abs{r^2-s^2}}\ dr\ ds.
\]
For $p=q=1$ this is majorized by
\[
\int_{\substack{0\le x,y\le 1\\xy\ge\delta}}
\frac1{xy}\ dx\ dy=
\int_\delta^1(\int_{\delta/x}^1\ \frac{dy}y)\ \frac{dx}x=
\int_\delta^1(\log x-\log\delta)\ \frac{dx}x=
\frac12(\log\delta)^2.
\]
Otherwise, it is majorized by
\[
\int_{\substack{0\le x,y\le 1\\xy\ge\delta}}
\frac{x+y}{xy}\ dx\ dy=
2\int_\delta^1(\int_{\delta/x}^1\ \frac{dy}y)\ dx=
2\int_\delta^1(\log x-\log\delta)\ dx\le-2\log\delta
\]
as required.
\end{proof}

This concludes the proof of Proposition \ref{maingeom}.

\section{The spectral side} \label{spec}
\setcounter{equation}{0}

In this section we analyze the spectral side of the trace formula for $G=\GL(n)$, following \cite{MR2276771}.
By \cite[Theorem 0.1]{MR2053600}, the spectral side $J_{\spec}$
is given as a finite linear combination
\[
J_{\spec}(f)=J_{\di}(f)+
\sum_{M\in\levis:M\ne G}\sum_{s\in W(\vs_M)} a_{M,s} J_{M,s}(f),\ \ \ 
f\in C_c^\infty(G(\A)^1),
\]
of distributions $J_{\di}$ and $J_{M,s}$.
To describe these distributions, we need to introduce some more notation.
Fix $M\in\levis$ and let $\pars(M)$ be the set of parabolic subgroups
of $G$ (necessarily defined over $\Q$) for which $M$ is a Levi component.
For $P\in\pars(M)$, we write $\vs_P=\vs_M$ and let $\AF^2(P)$ be the space of automorphic forms
on $N_P(\A)M_P(\Q)\bs G(\A)$ which are square integrable on 
$M_P(\Q)\bs M_P(\A)^1\times K$ \cite[p.~1249]{MR681737}. Given $Q\in\pars(M)$
let, $W(\vs_P,\vs_Q)$ be the set of all linear
isomorphisms from $\vs_P$ to $\vs_Q$ which are restrictions of elements of $W$. For $s\in W(\vs_P,\vs_Q)$ let
\[
M_{Q|P}(s,\lambda):\AF^2(P)\to\AF^2(Q),\quad\lambda\in\vs_{P,\C}^*,
\]
be the intertwining operator \cite[\S 1]{MR681738}, which is a meromorphic 
function
of $\lambda\in\vs_{P,\C}^*$. Set
\[
M_{Q|P}(\lambda):=M_{Q|P}(1,\lambda).
\]
Fix $P\in\pars(M)$  and $\lambda\in\iii\vs_M^*$. 
For $Q\in\pars(M)$ and $\Lambda\in\iii\vs_M^*$ define
\[
\mM_Q(P,\lambda,\Lambda)=M_{Q|P}(\lambda)^{-1}M_{Q|P}(\lambda+\Lambda).
\]
Then
\begin{equation}\label{GM}
\{\mM_Q(P,\lambda,\Lambda):\Lambda\in\iii\vs_M^*,\ Q\in\pars(M)\}
\end{equation}
is a $(G,M)$ family with values in the space of operators on $\AF^2(P)$
([ibid., p.~1310]). 
Let $L\in\levis$ with $L\supset M$. Then
the $(G,M)$ family (\ref{GM}) has an associated $(G,L)$ family
\[
\{\mM_{Q_1}(P,\lambda,\Lambda):\Lambda\in\iii\vs_L^*,\ Q_1\in\pars(L)\}
\]
[ibid., p.~1297]. For a parabolic group $P$ let 
\[
\theta_P(\lambda)=\vol\left(\vs_P/\Z(\Delta_P^\vee)\right)^{-1}
\prod_{\alpha\in\Delta_P}\lambda(\alpha^\vee),\quad\lambda\in\iii\vs_P^*,
\]
where $\Z(\Delta_P^\vee)$ is the lattice in $\vs_P$ generated by the 
co-roots $\alpha^\vee$, $\alpha\in\Delta_P$.
Then
\[
\mM_L(P,\lambda,\Lambda)=\sum_{Q_1\in\pars(L)}\mM_{Q_1}(P,\lambda,\Lambda)
\theta_{Q_1}(\Lambda)^{-1}
\]
extends to a smooth function on $\iii\vs_L^*$. In particular, set
\[
\mM_L(P,\lambda):=\mM_L(P,\lambda,0)=
\lim_{\substack{\Lambda\in\iii\vs_L^*\\\Lambda\to0}}\left(\sum_{Q_1\in\pars(L)}
\vol(\vs_{Q_1}/\Z(\Delta^\vee_{Q_1}))M_{Q|P}(\lambda)^{-1}
\frac{M_{Q|P}(\lambda+\Lambda)}{\prod_{\alpha\in\Delta_{Q_1}}
\Lambda(\alpha^\vee)}\right),
\]
where for each $Q_1\in\pars(L)$, $Q$ is a group in $\pars(M_P)$ which is contained in $Q_1$.
Let $\rho(P,\lambda)$ be the induced representation of $G(\A)$ on $\ov\AF^2(P)$.
Then the distribution $J_{M,s}$ is given by
\begin{equation} \label{absexp}
J_{M,s}(f)=\int_{\iii\vs_L^*}\tr\left(\mM_L(P,\lambda)M_{P|P}(s,0)
\rho(P,\lambda,f)\right)\ d\lambda
\end{equation}
where $L$ is the minimal Levi containing $M$ and $s$ and $P$ is any element of $\pars(M)$.
By \cite[Theorem 0.1]{MR2053600} this integral is absolutely convergent with 
respect to the trace norm.
(See \cite{FLM} for a more general result, where the independence of
the above expression on $P$ is also explained.)
Implicit here is that 
$\mM_L(P,\lambda)M_{P|P}(s,0)\rho(P,\lambda,f)$ extends to a trace-class
operator on $\ov\AF^2(P)$ for almost all $\lambda$.
Finally, $J_{\di}(f)$ is defined to be $J_{G,1}(f)$.

We will now show Proposition \ref{mainspec}, arguing as in \cite[\S5]{MR2276771}.
Fix $K_f$ and let $K=K_\infty K_f$.
We can assume without loss of generality that $h\in C_c^\infty(\vs)^W$. Let $f=\FFF(h)$.
We will expand (\ref{absexp}) according to $\pi\in\Pi_{\di}(M(\A)^1)$.
For each such $\pi=\pi_\infty\otimes\pi_f$ let 
$\AF^2_\pi(P)$ be the subspace of $\AF^2(P)$ of functions $\phi$ such that
for each $x\in G(\A)$, the function $\phi_x(m):=\phi(mx)$, $m\in M_P(\A)$,
transforms under $M_P(\A)$ according to the representation $\pi$. 
Let $M_{Q|P}(s,\pi,\lambda)$ and $\mM_L(P,\pi,\lambda)$ 
denote the restriction of $M_{Q|P}(s,\lambda)$ and  $\mM_L(P,\lambda)$, respectively,
to $\bar\AF^2_\pi(P)$.
Let $\AF^2_P(\pi)^K$ the subspace of $K$-invariant functions. 
This is a finite-dimensional space.
Let $\Pi_K$ be the orthogonal projection of $\overline{\AF^2_\pi(P)}$ onto
$\AF^2_P(\pi)^K$. Note that 
$\AF^2_P(\pi)^K=0$, unless the induced representation 
$I_{P(\R)}^{G(\R)}(\pi_\infty)$ has a non-zero $K_\infty$-fixed vector.
Let $\lambda_{\pi_\infty}\in(\vs^M)^*_\C/W_M$
be the infinitesimal character of $\pi_\infty$. Then
\[
\rho(P,\lambda,f)=\hat h(\lambda_{\pi_\infty}+\lambda)\Pi_K
\]
on $\AF^2_P(\pi)$ and by (\ref{absexp}) we get
\begin{equation}\label{morexp}
J_{M,s}(f)=\sum_{\pi\in\Pi_{\di}(M(\A)^1)}\int_{\iii\vs_L^*}\hat
h(\lambda_{\pi_\infty}+\lambda)\tr\left(\mM_L(P,\pi,\lambda)M_{P|P}
(s,\pi,0)\Pi_K\right)\ d\lambda. 
\end{equation}
In particular, for $M=G$ we obtain (\ref{Jdi})
since $L^2_{\di}(G(\Q)\bs G(\A)^1)$ is multiplicity free, 

To deal with the terms $M\ne G$ we first recall the normalized intertwining operators.
The normalizing factors $r_{Q|P}(\pi,\lambda)$, described in \cite[\S5]{MR2276771},
are given by
\[
r_{Q|P}(\pi,\lambda)=\prod_{\alpha\in\sum_P\cap\sum_{\overline Q}}
r_\alpha(\pi,\sprod{\lambda}{\alpha^\vee}),
\]
where each $r_\alpha(\pi,\cdot)$ is a meromorphic function on $\C$
which is given in terms of Rankin-Selberg $L$-functions. 
More precisely, if $M$ has the form
\[
M=\GL(n_1)\times\cdots\times\GL(n_r)
\]
with $n=n_1+\dots+n_r$ then we can write $\pi=\pi_1\otimes\cdots\otimes\pi_r$,
where $\pi_i\in\Pi_{\di}(\GL(n_i,\A))$, and a root $\alpha$ corresponds 
to an ordered pair $(i,j)$ of distinct integers between $1$ and $r$.
The normalizing factor $r_\alpha(\pi,s)$ is then given by
\[
r_\alpha(\pi,s)=\frac{L(s,\pi_i\times\widetilde\pi_j)}
{L(1+s,\pi_i\times\widetilde\pi_j)\epsilon(s,\pi_i\times\widetilde\pi_j)},
\]
where $L(s,\pi_i\times\widetilde\pi_j)$ is the global Rankin-Selberg
$L$-function  and $\epsilon(s,\pi_i\times\widetilde\pi_j)$  is the
corresponding $\epsilon$-factor. 

The normalized intertwining operator is defined by
\[
N_{Q|P}(\pi,\lambda):=r_{Q|P}(\pi,\lambda)^{-1}M_{Q|P}(\pi,\lambda),\quad
\lambda\in\vs_{M,\C}^*.
\]
For a Levi subgroup $L$ let $\leviP(L)$ be the set of all parabolic subgroups
containing $L$. Using the basic properties of $(G,M)$ families
(\cite[p.~1329]{MR681738}), one gets 
\[
\mM_L(P,\pi,\lambda)=\sum_{S\in\leviP(L)}\mN_S'(P,\pi,\lambda)
\nu_L^S(P,\pi,\lambda),
\]
where the operator $\mN_S'(P,\pi,\lambda)$ is obtained from the $(G,L)$ 
family attached to $N_{Q|P}(\pi,\lambda)$ (cf.~\cite[(6.10)]{MR2276771})
and $\nu_L^S(P,\pi,\lambda)$ is defined in terms of the normalizing factors
(see below).

Therefore, by (\ref{morexp}) $J_{M,s}(f)$ equals
\[
\sum_{\pi\in\Pi_{\di}(M(\A)^1)}
\sum_{S\in\leviP(L)}\int_{\iii\vs_L^*}\hat
h(\lambda_{\pi_\infty}+\lambda)\nu_L^S(P,\pi,\lambda)
\tr\left(M_{P|P}(s,\pi,0)\mN_S'(P,\pi,\lambda)_K\right)\ d\lambda. 
\]
where $\mN_S'(P,\pi,\lambda)_K$ denotes the
restriction of $\mN_S'(P,\pi,\lambda)$ to $\AF^2_\pi(P)^K$.
By \cite[(6.13)]{MR2276771}
\[
\norm{\mN_S'(P,\pi,\lambda)_K}=O_K(1)
\]
for all $\lambda\in\iii\vs_M^*$ and $\pi\in\Pi_{\di}(M(\A)^1)$. Since 
$M_{P|P}(s,\pi,0)$ is unitary, $J_{M,s}(f)$ is bounded by a constant
multiple of
\begin{equation}\label{Jbnd1}
\sum_{\pi\in\Pi_{\di}(M(\A)^1)}
\sum_{S\in\leviP(L)}\dim\left(\AF^2_\pi(P)^K\right)\int_{\iii\vs_L^*}
\abs{\hat h(\lambda_{\pi_\infty}+\lambda)}\cdot
\abs{\nu_L^S(P,\pi,\lambda)}\ d\lambda. 
\end{equation}

The function $\nu_L^S(P,\pi,\lambda)$ can be described as follows. If $F$ is a
subset of $\Sigma(G,A_M)$, let $F_L^\vee=\{\alpha_L^\vee:\alpha\in F\}$ (as a multiset). Suppose that $S\in\pars(L_1)$ with $L_1\supset L$. Then by \cite[Proposition 7.5]{MR681738} we have
\[
\nu_L^S(P,\pi,\lambda)=\sum_F\vol(\vs_L^{L_1}/\Z(F_L^\vee))
\cdot\prod_{\alpha\in F}
\frac{r_\alpha'(\pi,\sprod{\lambda}{\alpha^\vee})}
{r_\alpha(\pi,\sprod{\lambda}{\alpha^\vee})}
\]
where $F$ runs over all subsets of $\Sigma(L_1,A_M)$ such that $F_L^\vee$ 
is a basis of $\vs_L^{L_1}$.
The argument of \cite[Proposition 5.1]{MR2276771} in conjunction with 
[ibid., Lemmas 5.3 and 5.4] gives
\[
\int_T^{T+1}\abs{\frac{r'(\pi_1\otimes\pi_2,\iii t)}{r(\pi_1\otimes\pi_2,\iii t)}}\ dt
=O\left(\log(\abs{T}+\norm{\lambda_{\pi_{1,\infty}}}+
\norm{\lambda_{\pi_{2,\infty}}}+2)\right)
\]
for all $\pi_i\in\Pi_{\di}(\GL(m_i,\A))$, $i=1,2$ and any $T\in\R$.
Therefore,
\begin{equation} \label{avenu}
\int_{B(\mu)\cap \iii\vs_L^*}\abs{\nu_L^S(P,\pi,\lambda)}\ d\lambda
=O\left(\log^l(\norm{\lambda_{\pi_\infty}+\mu}+2)\right)
\end{equation}
for any $\pi\in\Pi_{\di}(M(\A)^1)$ and $\mu\in\iii\vs_L^*$ where 
$l=\dim\vs_L^*$.
On the other hand, 
upon replacing $K_f$ by an open subgroup which is normal in $K_f(1)=\prod_{p<\infty}\GL(n,\Z_p)$
we can use (\cite[(3.5) and (3.7)]{MR2276771}) to obtain
\begin{equation} \label{multM}
\sum_{\pi\in\Pi_{\di}(M(\A)^1):\lambda_{\pi_\infty}=\lambda}
\dim\left(\AF^2_\pi(P)^K\right)\le [K_f(1):K_f]m_M^{K_f\cap M(\A_f)}(\lambda)
\end{equation}
for any $\lambda\in(\vs^M_\C)^*$
where $m_M^{K_f\cap M(\A_f)}$ measures the multiplicity
with respect to the locally symmetric space $M(F)\bs M(\A)^1/(K_f\cap M(\A_f))$.

We are now ready to prove Proposition \ref{mainspec} by induction on $n$. For
$\mu\in (\vs^M_\C)^*$ let $B^M(\mu)$ be the ball in $(\vs^M_\C)^*$ of 
radius $1$ with center $\mu$.
Using (\ref{multM}) and taking into account Remark \ref{rem10} 
(end of section 4) we can use the induction hypothesis
to apply Proposition \ref{locuprbnd} to $M$ to infer that
\begin{equation} \label{frob}
\sum_{\pi\in\Pi_{\di}(M(\A)^1):\Im\lambda_{\pi_\infty}\in B^M(\mu)}
\dim\left(\AF^2_\pi(P)^K\right)=O(\varbeta^M(\mu))
\end{equation}
for any $\mu\in(\vs^M)^*$.

Let $\vs_M^L=\vs_M\cap\vs^L$ and
denote by $(\vs_M^L)^\perp=(\vs^M)^*\oplus\vs_L^*$ its annihilator in $\vs^*$.
Let $\Latt$ be a lattice in $\iii(\vs_M^L)^\perp$ such that the balls
$B(\mu)\cap \iii(\vs_M^L)^\perp$, $\mu\in\Latt$ cover $\iii(\vs_M^L)^\perp$.
We estimate (\ref{Jbnd1}) by splitting the sum and the integral to the sets
$\Im\lambda_{\pi_\infty}+\lambda\in B(\mu)\cap \iii(\vs_M^L)^\perp$,
$\mu\in\Latt$, and recalling that $\Re\lambda_{\pi_\infty}\le\norm{\rho^M}$.
We obtain that (\ref{Jbnd1}) is bounded by
\[
\sum_{\mu\in\Latt}
\max_{B_{1+\norm{\rho^M}}(\mu)}\abs{\hat h}\sum_{S\in\leviP(L)}
\sum_{\substack{\pi\in\Pi_{\di}(M(\A)^1)\\\Im\lambda_{\pi_\infty}\in
B^M(\mu^M)}}
\dim\left(\AF^2_\pi(P)^K\right)
\int_{B(\mu_L)\cap \iii\vs_L^*}\abs{\nu_L^S(P,\pi,\lambda)}\ d\lambda.
\]
By (\ref{avenu}) and (\ref{frob}) this is majorized by
\begin{equation} \label{lstbnd}
\sum_{\mu\in\Latt}\max_{B_{1+\norm{\rho^M}}(\mu)}\abs{\hat h}
\ \varbeta^M(\mu^M)\log^l(\norm{\mu}+2).
\end{equation}
Using Lemma \ref{scr}, the series can be estimated by $O(\nm (h))$. 
Thus, $J_{M,s}(\FFF(h))=O(\nm (h))$.
This gives the first part of Proposition \ref{mainspec}.
To show the second part note that for $M\neq G$ we have 
\[
\varbeta^M(\mu^M)=\begin{cases}1,&\text{if }n=2;
\\O\left((\norm{\mu}+1)^{d-r-2}\right),&
\text{otherwise}.\end{cases}
\]
Thus, by (\ref{lstbnd}) we have
\[
\int_{t\Omega}J_{M,s}(\FFF(h_\mu))\ d\mu=
\begin{cases}O_{h,\Omega}(t^{d-2}\log^r(t+2)),&\text{if }n>2,\\
O_{h,\Omega}(t^{d-1}\log(t+2)),&\text{if }n=2.\end{cases}
\]
This concludes the proof of Proposition \ref{mainspec}.

%\bibliographystyle{alpha}
%\bibliography{../Bibfiles/all}
%\bibliography{all}

\begin{thebibliography}{GKM97}

\bibitem[AG91]{MR1110389}
James Arthur and Stephen Gelbart.
\newblock Lectures on automorphic {$L$}-functions.
\newblock In {\em $L$-functions and arithmetic (Durham, 1989)}, volume 153 of
  {\em London Math. Soc. Lecture Note Ser.}, pages 1--59. Cambridge Univ.
  Press, Cambridge, 1991.

\bibitem[Art78]{MR518111}
James~G. Arthur.
\newblock A trace formula for reductive groups. {I}. {T}erms associated to
  classes in {$G({\bf Q})$}.
\newblock {\em Duke Math. J.}, 45(4):911--952, 1978.

\bibitem[Art82a]{MR681737}
James Arthur.
\newblock On a family of distributions obtained from {E}isenstein series. {I}.
  {A}pplication of the {P}aley-{W}iener theorem.
\newblock {\em Amer. J. Math.}, 104(6):1243--1288, 1982.

\bibitem[Art82b]{MR681738}
James Arthur.
\newblock On a family of distributions obtained from {E}isenstein series. {II}.
  {E}xplicit formulas.
\newblock {\em Amer. J. Math.}, 104(6):1289--1336, 1982.

\bibitem[Art85]{MR828844}
James Arthur.
\newblock A measure on the unipotent variety.
\newblock {\em Canad. J. Math.}, 37(6):1237--1274, 1985.

\bibitem[Art88]{MR932848}
James Arthur.
\newblock The local behaviour of weighted orbital integrals.
\newblock {\em Duke Math. J.}, 56(2):223--293, 1988.

\bibitem[Art89]{MR1001841}
James Arthur.
\newblock The {$L\sp 2$}-{L}efschetz numbers of {H}ecke operators.
\newblock {\em Invent. Math.}, 97(2):257--290, 1989.

\bibitem[Ava56]{MR0080862}
 Vojislav G. Avakumovi\'c.
\newblock \"Uber die Eigenfunktionen auf geschlossenen Riemannschen 
Mannigfaltigkeiten.
\newblock {\em Math. Z.}, 65:327--344, 1956.


\bibitem[Bor69]{MR0244260}
Armand Borel.
\newblock {\em Introduction aux groupes arithm\'etiques}.
\newblock Publications de l'Institut de Math\'ematique de l'Universit\'e de
  Strasbourg, XV. Actualit\'es Scientifiques et Industrielles, No. 1341.
  Hermann, Paris, 1969.

\bibitem[CM93]{MR1251060}
David~H. Collingwood and William~M. McGovern.
\newblock {\em Nilpotent orbits in semisimple {L}ie algebras}.
\newblock Van Nostrand Reinhold Mathematics Series. Van Nostrand Reinhold Co.,
  New York, 1993.

\bibitem[DG75]{MR0423438}
J.~J. Duistermaat and V.~W. Guillemin.
\newblock The spectrum of positive elliptic operators and periodic geodesics.
\newblock In {\em Differential geometry (Proc. Sympos. Pure Math., Vol. XXVII,
  Part 2, Stanford Univ., Stanford, Calif., 1973)}, pages 205--209. Amer. Math.
  Soc., Providence, R. I., 1975.

\bibitem[DKV79]{MR532745}
J.~J. Duistermaat, J.~A.~C. Kolk, and V.~S. Varadarajan.
\newblock Spectra of compact locally symmetric manifolds of negative curvature.
\newblock {\em Invent. Math.}, 52(1):27--93, 1979.

\bibitem[DKV83]{MR707179}
J.~J. Duistermaat, J.~A.~C. Kolk, and V.~S. Varadarajan.
\newblock Functions, flows and oscillatory integrals on flag manifolds and
  conjugacy classes in real semisimple {L}ie groups.
\newblock {\em Compositio Math.}, 49(3):309--398, 1983.

\bibitem[DKV84]{MR771672}
P.~Deligne, D.~Kazhdan, and M.-F. Vign{\'e}ras.
\newblock Repr\'esentations des alg\`ebres centrales simples {$p$}-adiques.
\newblock In {\em Representations of reductive groups over a local field},
  Travaux en Cours, pages 33--117. Hermann, Paris, 1984.

\bibitem[FLM]{FLM}
Tobias Finis, Erez Lapid, and Werner M\"uller.
\newblock On the spectral side of {A}rthur's trace formula {I}{I}.
\newblock {\em preprint}.

\bibitem[GKM97]{MR1470341}
M.~Goresky, R.~Kottwitz, and R.~MacPherson.
\newblock Discrete series characters and the {L}efschetz formula for {H}ecke
  operators.
\newblock {\em Duke Math. J.}, 89(3):477--554, 1997.

\bibitem[Hej76]{MR0439755}
Dennis~A. Hejhal.
\newblock {\em The {S}elberg trace formula for {${\rm PSL}(2,R)$}. {V}ol. {I}}.
\newblock Springer-Verlag, Berlin, 1976.
\newblock Lecture Notes in Mathematics, Vol. 548.

\bibitem[Hel00]{MR1790156}
Sigurdur Helgason.
\newblock {\em Groups and geometric analysis}, volume~83 of {\em Mathematical
  Surveys and Monographs}.
\newblock American Mathematical Society, Providence, RI, 2000.
\newblock Integral geometry, invariant differential operators, and spherical
  functions, Corrected reprint of the 1984 original.

\bibitem[H{\"o}r68]{MR0609014}
Lars H{\"o}rmander.
\newblock The spectral function of an elliptic operator.
\newblock {\em Acta Math.}, 121:193--218, 1968.

\bibitem[ILS00]{MR1828743}
Henryk Iwaniec, Wenzhi Luo, and Peter Sarnak.
\newblock Low lying zeros of families of {$L$}-functions.
\newblock {\em Inst. Hautes \'Etudes Sci. Publ. Math.}, (91):55--131 (2001),
  2000.

\bibitem[JS81a]{MR623137}
H.~Jacquet and J.~A. Shalika.
\newblock On {E}uler products and the classification of automorphic forms.
  {II}.
\newblock {\em Amer. J. Math.}, 103(4):777--815, 1981.

\bibitem[JS81b]{MR618323}
H.~Jacquet and J.~A. Shalika.
\newblock On {E}uler products and the classification of automorphic
  representations. {I}.
\newblock {\em Amer. J. Math.}, 103(3):499--558, 1981.

\bibitem[Lan63]{MR0156362}
R.~P. Langlands.
\newblock The dimension of spaces of automorphic forms.
\newblock {\em Amer. J. Math.}, 85:99--125, 1963.

\bibitem[LM04]{MR2127945}
Jean-Pierre Labesse and Werner M{\"u}ller.
\newblock Weak {W}eyl's law for congruence subgroups.
\newblock {\em Asian J. Math.}, 8(4):733--745, 2004.

\bibitem[LV07]{LV}
Elon Lindenstrauss and Akshay Venkatesh.
\newblock Existence and {W}eyl's law for spherical cusp forms.
\newblock {\em Geom. Funct. Anal.}, 17(1):220--251, 2007.

\bibitem[Mil01]{MR1823867}
Stephen~D. Miller.
\newblock On the existence and temperedness of cusp forms for {${\rm SL}\sb
  3({\Bbb Z})$}.
\newblock {\em J. Reine Angew. Math.}, 533:127--169, 2001.

\bibitem[MS04]{MR2053600}
W.~M{\"u}ller and B.~Speh.
\newblock Absolute convergence of the spectral side of the {A}rthur trace
  formula for {${\rm GL}\sb n$}.
\newblock {\em Geom. Funct. Anal.}, 14(1):58--93, 2004.
\newblock With an appendix by E. M.\ Lapid.

\bibitem[M{\"u}l]{Mu}
Werner M{\"u}ller.
\newblock Weyl's law in the thoery of automorphic forms.
\newblock In {\em Groups and {A}nalysis: {T}he {L}egacy of {H}ermann {W}eyl}.
  Cambridge Univ. Press, to appear, arXiv:0710.2319.

\bibitem[M{\"u}l07]{MR2276771}
Werner M{\"u}ller.
\newblock Weyl's law for the cuspidal spectrum of {${\rm SL}\sb n$}.
\newblock {\em Ann. of Math. (2)}, 165(1):275--333, 2007. 

\bibitem[Mum70]{MR0282985}
David Mumford.
\newblock {\em Abelian varieties}.
\newblock Tata Institute of Fundamental Research Studies in Mathematics, No. 5.
  Published for the Tata Institute of Fundamental Research, Bombay, 1970.

\bibitem[MW89]{MR1026752}
C.~M{\oe}glin and J.-L. Waldspurger.
\newblock Le spectre r\'esiduel de {${\rm GL}(n)$}.
\newblock {\em Ann. Sci. \'Ecole Norm. Sup. (4)}, 22(4):605--674, 1989.

\bibitem[OT92]{MR1217488}
Peter Orlik and Hiroaki Terao.
\newblock {\em Arrangements of hyperplanes}, volume 300 of {\em Grundlehren der
  Mathematischen Wissenschaften [Fundamental Principles of Mathematical
  Sciences]}.
\newblock Springer-Verlag, Berlin, 1992.

\bibitem[PS85]{MR812352}
R.~S. Phillips and P.~Sarnak.
\newblock The {W}eyl theorem and the deformation of discrete groups.
\newblock {\em Comm. Pure Appl. Math.}, 38(6):853--866, 1985.

\bibitem[PS92]{MR1127079}
R.~Phillips and P.~Sarnak.
\newblock Perturbation theory for the {L}aplacian on automorphic functions.
\newblock {\em J. Amer. Math. Soc.}, 5(1):1--32, 1992.

\bibitem[Rez93]{MR1204788}
Andrei Reznikov.
\newblock Eisenstein matrix and existence of cusp forms in rank one symmetric
  spaces.
\newblock {\em Geom. Funct. Anal.}, 3(1):79--105, 1993.

\bibitem[Sar86]{MR853570}
Peter Sarnak.
\newblock On cusp forms.
\newblock In {\em The Selberg trace formula and related topics (Brunswick,
  Maine, 1984)}, volume~53 of {\em Contemp. Math.}, pages 393--407. Amer. Math.
  Soc., Providence, RI, 1986.

\bibitem[Sel56]{MR0088511}
A.~Selberg.
\newblock Harmonic analysis and discontinuous groups in weakly symmetric
  {R}iemannian spaces with applications to {D}irichlet series.
\newblock {\em J. Indian Math. Soc. (N.S.)}, 20:47--87, 1956.

\bibitem[Sel89]{MR1117906}
Atle Selberg.
\newblock {\em Collected papers. {V}ol. {I}}.
\newblock Springer-Verlag, Berlin, 1989.
\newblock With a foreword by K. Chandrasekharan.

\end{thebibliography}

\end{document}